\documentclass[11pt]{amsart}
\usepackage[T1]{fontenc}
\usepackage[utf8]{inputenc}
\usepackage{lmodern}
\usepackage{microtype}
\usepackage{xcolor}
\usepackage{amsmath,amssymb,amsthm,mathtools}
\usepackage{enumitem}
\usepackage[margin=1in]{geometry}
\usepackage{hyperref}
\hypersetup{colorlinks=true,linkcolor=blue,citecolor=blue,urlcolor=blue}
\usepackage{mathrsfs}

\theoremstyle{plain}
\newtheorem{theorem}{Theorem}[section]
\newtheorem{lemma}[theorem]{Lemma}
\newtheorem{proposition}[theorem]{Proposition}
\newtheorem{corollary}[theorem]{Corollary}
\theoremstyle{remark}
\newtheorem{remark}[theorem]{Remark}
\newtheorem{example}[theorem]{Example}

\DeclareMathOperator{\rank}{rank}
\DeclareMathOperator{\diag}{diag}
\newcommand{\C}{\mathbb{C}}
\numberwithin{equation}{section}

\newcommand{\N}{\mathbb{N}}
\newcommand{\A}{\mathcal{A}}
\newcommand{\B}{\mathcal{B}}
\newcommand{\J}{\mathcal{J}}
\newcommand{\K}{\mathcal{K}}
\newcommand{\Pol}{\mathscr{P}}
\DeclareMathOperator{\spc}{sp}
\DeclareMathOperator{\Poly}{Pol}
\DeclareMathOperator{\Nil}{Nil}
\DeclareMathOperator{\rad}{rad}
\DeclareMathOperator{\disc}{disc}

\hypersetup{pdftitle={Continuous spectrum-shrinking maps between finite-dimensional power-associative algebras},pdfauthor={Ilja Gogic},pdfsubject={Continuous spectrum-shrinking maps between power-associative algebras}}

\title[Spectrum-shrinking maps between power-associative algebras]{Continuous spectrum-shrinking maps between finite-dimensional power-associative algebras}

\author{Ilja Gogi\'{c}}

\address{I.~Gogi\'c, University of Zagreb Faculty of Science, Department of Mathematics, Bijeni\v{c}ka 30, 10000 Zagreb, Croatia}

\email{ilja@math.hr, ilja.gogic@math.pmf.unizg.hr}

\date{\today}



\begin{document}

\begin{abstract}
For finite-dimensional complex power-associative algebras $\mathcal{A}$ and $\mathcal{B}$, each equipped with its unique Hausdorff vector topology, we characterize continuous spectrum-shrinking maps $\phi:\mathcal{A}\to\mathcal{B}$, i.e.\ maps satisfying $\mathrm{sp}_{\mathcal{B}}(\phi(x)) \subseteq\mathrm{sp}_{\mathcal{A}}(x)$ for all $x\in\mathcal{A}$. The spectrum of an element is the zero set of its minimal polynomial, computed in the unitization for nonunital algebras. No such map exists when $\mathcal{A}$ is unital and $\mathcal{B}$ is nonunital. Otherwise, we symmetrize both products by setting $x\circ y:=\tfrac{1}{2}(xy+yx)$ and factor out the nilradicals to obtain semisimple Jordan quotients. We prove that a continuous map $\phi$ is spectrum-shrinking precisely when, for every $x\in\mathcal{A}$, Jacobson's generic minimal polynomials of the simple quotient coordinates of $\phi(x)$ are products of fixed nonnegative integer powers of those of $x$, with an additional fixed monomial factor allowed when $\mathcal{A}$ is nonunital. Comparing degrees gives linear Diophantine equations characterizing existence, and every nonnegative integer solution is realized by a polynomial spectrum-shrinking map. Spectrum preservation, meaning equality of spectra, is equivalent to additional positivity conditions. We also determine when every continuous spectrum-shrinking map preserves spectra and show that every continuous map $f:\mathcal{A}\to\mathbb{C}$ with $f(x)\in\mathrm{sp}_{\mathcal{A}}(x)$ for all $x\in\mathcal{A}$ is $\mathbb{C}$-linear and satisfies $f(x\circ y)=f(x)f(y)$ for all $x,y\in\mathcal{A}$.
\end{abstract}

\maketitle

\section{Introduction}

Spectral conditions play a central role in preserver problems on matrix and operator algebras. For a unital finite-dimensional complex associative algebra $\A$ with identity $1_\A$, the spectrum of $x\in\A$ is
\begin{equation}\label{eq:spc_assoc}
 \spc_\A(x):=\{\lambda\in\C:\lambda 1_\A-x\text{ is not invertible in }\A\}.
\end{equation}
Given two such algebras $\A$ and $\B$, a map $\phi:\A\to\B$ is called \emph{spectrum-shrinking} if
\begin{equation}\label{eq:spectrum-shrinking}
 \spc_\B(\phi(x))\subseteq\spc_\A(x),\qquad x\in\A,
\end{equation}
and \emph{spectrum-preserving} if equality holds in~\eqref{eq:spectrum-shrinking} for every $x\in\A$. For unital linear maps, these conditions are equivalent to preservation of invertibility in one direction and in both directions, respectively.

Unital Jordan homomorphisms provide basic examples of spectrum-shrinking maps. Recall that every complex associative algebra $\A$ becomes a Jordan algebra $\A^+$ under the \emph{symmetrized product}
\begin{equation}\label{eq:symmetrized-product}
 x\circ y:=\frac{1}{2}(xy+yx),\qquad x,y\in\A;
\end{equation}
see~\cite[Chapter~I]{Schafer1966}. A \emph{Jordan homomorphism} between complex associative algebras $\A$ and $\B$ is a linear map $\phi:\A\to\B$ that preserves squares or, equivalently, a homomorphism $\A^+\to\B^+$. Conversely, building on Frobenius's classification of linear determinant preservers~\cite{Frobenius1897}, Marcus and Purves~\cite[Theorem~2.1]{MarcusPurves1959} proved that, for each positive integer $n$, every unital linear spectrum-shrinking self-map of the algebra $M_n(\C)$ of all $n\times n$ complex matrices is a Jordan automorphism.

These results lead to the Kaplansky--Aupetit invertibility-preserver problem~\cite{Aupetit2000,Kaplansky1970}: must every surjective unital linear spectrum-shrinking map between unital semisimple complex Banach algebras be a Jordan homomorphism? Sourour proved that unital linear spectrum-shrinking bijections between the algebras of all bounded linear operators on complex Banach spaces are Jordan isomorphisms~\cite[Theorem~1.1]{Sourour1996}. Aupetit obtained the same conclusion for surjective linear spectrum-preserving maps between von Neumann algebras~\cite[Theorem~1.3]{Aupetit2000}. Cui and Hou subsequently settled the spectrum-shrinking question for unital linear maps from a von Neumann algebra onto an arbitrary unital semisimple Banach algebra~\cite{CuiHou2004}. The problem remains open for general unital $C^*$-algebras~\cite[Introduction]{MathieuSchulz2023}. In finite dimensions, the theorem of Mathieu and Schulz~\cite[Theorem]{MathieuSchulz2023} gives an affirmative answer without requiring the domain to be semisimple: if $\A$ and $\B$ are unital finite-dimensional complex associative algebras with $\B$ semisimple, then every surjective unital linear spectrum-shrinking map $\A\to\B$ is a Jordan homomorphism. Here $\B$ is isomorphic to a finite direct sum of full matrix algebras, in particular to a $C^*$-algebra with a faithful tracial state.

Spectrum preservation has also been studied for nonlinear maps on the matrix algebra $M_n(\C)$. Baribeau and Ransford proved that, for $n\ge2$, every spectrum-preserving $C^1$-diffeomorphism between open subsets of $M_n(\C)$ maps each matrix to a similar one~\cite[Corollary~1.2]{BaribeauRansford2000}. Costara~\cite[Theorem~1]{Costara2011} proved that, for $n\ge2$, every continuous map $\phi:M_n(\C)\to M_n(\C)$ satisfying $\phi(0)=0$ and $\spc_{M_n(\C)}(\phi(A)-\phi(B))\subseteq\spc_{M_n(\C)}(A-B)$ for all $A,B\in M_n(\C)$ is a Jordan automorphism. Dolinar, Hou, Kuzma and Qi removed the continuity assumption~\cite[Theorem~1.1]{DHKQ2013}. For $n\ge3$, \v{S}emrl characterized the Jordan automorphisms of $M_n(\C)$ as its continuous self-maps preserving spectrum and commutativity~\cite[Theorem~1.1]{Semrl2008}. Recently, Chirvasitu studied continuous maps preserving spectrum and commutativity on special unitary groups, normal matrices with spectra constrained to curves, and semisimple matrices~\cite{Chirvasitu2025b,Chirvasitu2026a,Chirvasitu2026b}, as well as continuous eigenvalue selection for representations of compact connected groups~\cite{Chirvasitu2025a}.

Moreover, continuity alone imposes strong restrictions on spectrum-shrinking maps. In joint work with Chirvasitu and Toma\v{s}evi\'c, we proved that, for positive integers $m$ and $n$,
a continuous spectrum-shrinking map $\phi:M_n(\C)\to M_m(\C)$ exists if and only if $n$ divides $m$, in which case the characteristic polynomials satisfy
\begin{equation}\label{eq:matrix-multiplicities}
 \det(tI_m-\phi(A))=\det(tI_n-A)^{m/n},
 \qquad A\in M_n(\C),
\end{equation}
where $t$ is an indeterminate and $I_m$ and $I_n$ are the corresponding identity matrices.
In particular, every such map is spectrum-preserving. The same conclusions hold when the domain is the general linear, special linear or unitary group, the set of normal matrices, or the set of diagonalizable matrices in $M_n(\C)$ or in either of the first two groups~\cite[Theorem~1.1 and Corollary~2.2]{CGT2026}. For positive integers $n$ and $k$ with $n-\sqrt n<k<n$, we also proved that every injective continuous spectrum-shrinking map from the matrices of rank at most $k$ in $M_n(\C)$ into $M_n(\C)$ preserves characteristic polynomials~\cite[Theorem~1.2(4)]{CGT2025}.
Here injectivity is essential, as the zero map shows.

In subsequent work with Toma\v{s}evi\'c~\cite{GT2025}, we characterized the existence of continuous spectrum-shrinking and spectrum-preserving maps between unital finite-dimensional complex associative algebras in terms of nonnegative integer solutions of a linear Diophantine system determined by the sizes of their simple matrix quotients. We also showed that if every continuous spectrum-shrinking map is spectrum-preserving, then every such solution assigns a positive multiplicity to each simple source quotient in at least one target quotient.
We proved the converse for sources isomorphic to structural matrix algebras, that is, subalgebras of full matrix algebras containing all diagonal matrices. The general case remained open~\cite[Theorem~1.1(c) and the subsequent discussion]{GT2025}.

In this paper we extend this analysis to arbitrary finite-dimensional complex power-associative algebras. Recall that an algebra is \emph{power-associative} if each element generates an associative subalgebra~\cite[p.~552]{Albert1948}, so powers of an element and polynomial evaluation at it are unambiguous. This class includes Jordan and alternative algebras and, more generally, all noncommutative Jordan algebras~\cite[Chapter~V, \S\,3]{Schafer1966}. For such an algebra $\A$, the spectrum $\spc_\A(x)$ of $x\in\A$ is defined as the zero set of its minimal polynomial, computed in $\A$ when $\A$ is unital and in its unitization otherwise. For unital associative and Jordan algebras, this agrees with the usual spectrum; see Section~\ref{subsec:calculus}. We use the same definitions of spectrum-shrinking and spectrum-preserving maps in this setting.

Symmetrization preserves powers and spectra, so we first replace $\A$ by $\A^+$, with the
product~\eqref{eq:symmetrized-product}. This algebra is commutative and power-associative, but need not be a Jordan algebra~\cite[p.~504]{Albert1950}. We therefore pass to the quotient $\widehat\A$ of $\A^+$ by its nilradical, the largest ideal consisting of nilpotent elements. This ideal may be nonzero even when $\A$ is simple, as shown in Example~\ref{ex:symmetrization}. By results of Albert~\cite[Theorem~8]{Albert1950} and Kokoris~\cite[Theorem]{Kokoris1956}, $\widehat\A$ is a semisimple Jordan algebra and hence a finite direct sum of simple Jordan algebras.

Each simple Jordan algebra carries Jacobson's generic minimal polynomial~\cite{Jacobson1963}. Its degree is the Jordan rank of the algebra, and evaluating its coefficients at an element gives a polynomial whose roots form the spectrum of that element; see Subsection~\ref{subsec:generic}. Lemma~\ref{lem:reduction} shows that the spectrum of $x\in\A$ is the union of the spectra of the simple coordinates of its image in $\widehat\A$, with $0$ adjoined when $\A$ is nonunital. For $x\in\A$ and $y\in\B$, we call the generic minimal polynomials of their simple quotient coordinates the \emph{source factors} at $x$ and the \emph{target factors} at $y$, respectively.

Our main result, Theorem~\ref{thm:classification}, characterizes continuous spectrum-shrinking maps by a product formula for their target factors. No spectrum-shrinking map exists from a unital source to a nonunital target. Otherwise, a continuous map $\phi:\A\to\B$ is spectrum-shrinking if and only if, for every $x\in\A$, each target factor at $\phi(x)$ is a product of source factors at $x$ with fixed nonnegative integer exponents, together with a fixed nonnegative power of the polynomial indeterminate $t$ when $\A$ is nonunital. Comparing degrees gives a system of linear Diophantine equations involving the Jordan ranks. Such a map exists exactly when this system has a solution in nonnegative integers, and every such solution is realized by a polynomial spectrum-shrinking map. Spectrum preservation holds precisely when every source factor occurs with positive exponent in at least one target product, with the same requirement for the additional factor $t$ when $\A$ is nonunital and $\B$ is unital. For unital associative algebras, these equations coincide with those of~\cite[Theorem~1.1]{GT2025}.

The classification also determines when every continuous spectrum-shrinking map preserves spectra (Corollaries~\ref{cor:automatic} and~\ref{cor:nonunital}). In particular, Corollary~\ref{cor:associative} proves the converse of~\cite[Theorem~1.1(c)]{GT2025} for arbitrary unital associative sources. For a nonunital source, every spectrum-shrinking map preserves spectra if and only if every element of the source is nilpotent. For $\B=\C$, Corollary~\ref{cor:selectors} classifies the continuous maps $f:\A\to\C$ satisfying $f(x)\in\spc_\A(x)$ for every $x\in\A$. These are exactly the nonzero linear multiplicative maps $\A^+\to\C$, together with the zero map when $\A$ is nonunital. When $\A$ is unital and associative, this is a nonlinear counterpart of the finite-dimensional Gleason--Kahane--\.Zelazko theorem~\cite{Gleason1967,KahaneZelazko1968,Zelazko1968}: continuity and pointwise spectral selection force linearity and multiplicativity. In the same setting, Kowalski and S\l{}odkowski~\cite{KowalskiSlodkowski1980} obtain the same conclusion from $f(0)=0$ and $f(x)-f(y)\in\spc_\A(x-y)$ for all $x,y\in\A$, without assuming continuity.

The proof of Theorem~\ref{thm:classification} combines polynomial sections of the generic minimal polynomial maps (Proposition~\ref{prop:sections}) with a permutation argument adapted from~\cite[proof of Theorem~2.1]{CGT2026}. Together with linear right inverses of the quotient maps, these sections produce elements of $\A$ and $\B$ with prescribed source and target factors, respectively. The source sections parametrize a family of elements of $\A$ by ordered tuples of pairwise distinct source roots, required to be nonzero when $\A$ is nonunital. On this connected parameter space, continuity makes the multiplicity of each root in each target factor constant. Permuting the roots within a source factor leaves the corresponding element of $\A$ unchanged, so these roots have equal multiplicities in each target factor. Rouch\'e's theorem and connectedness then extend the resulting product formula to the connected dense open subset of $\A$ where the product of the source factors, with an additional factor $t$ when $\A$ is nonunital, has only simple roots. Continuity gives the formula on all of $\A$ (Lemma~\ref{lem:multiplicities}).

The paper is organized as follows. Section~\ref{sec:preliminaries} develops the spectral and
Jordan-algebraic background and proves the spectral reduction. Section~\ref{sec:main} constructs polynomial sections and proves the main theorem. Section~\ref{sec:consequences} gives criteria for automatic spectrum preservation, classifies continuous scalar spectral selectors and specializes the results to associative algebras. It concludes with examples illustrating the roles of symmetrization, continuity and the nilradical.

\section{Preliminaries}\label{sec:preliminaries}

Throughout, $\A$, $\B$ and all other algebras whose spectra we consider are finite-dimensional complex power-associative algebras.

An algebra is \emph{unital} if it has a nonzero two-sided identity. In particular, the zero algebra is nonunital. We write $1$ for the identity when the algebra is clear from the context and use a subscript otherwise. All linear maps and homomorphisms between complex vector spaces and algebras are understood to be complex-linear. A homomorphism between unital algebras is \emph{unital} if it preserves their identities. A \emph{character} is a nonzero homomorphism into $\C$.

By an ideal we always mean a two-sided ideal. An element $x\in\A$ is \emph{nilpotent} if $x^k=0$ for some positive integer $k$. An algebra or ideal is \emph{nil} if all its elements are nilpotent. An algebra is \emph{simple} if its product is not identically zero and its only ideals are $\{0\}$ and the algebra itself.

A \emph{Jordan algebra} is an algebra $\J$ whose multiplication $\circ$ is commutative and satisfies the \emph{Jordan identity}
\[
 x^2\circ(x\circ y)=x\circ(x^2\circ y), \qquad x,y\in\J.
\]
Every Jordan algebra is power-associative~\cite[Part~II, Theorem~5.2.2(1)]{McCrimmon2004}. A homomorphism between Jordan algebras is called a \emph{Jordan homomorphism}. Standard references on Jordan algebras are~\cite{Jacobson1968,McCrimmon2004}.

We write $\N$ and $\N_0$ for the sets of positive and nonnegative integers, respectively. By $\C[t]$ we denote the ring of complex polynomials in the indeterminate $t$. For $n\in\N$, let
\[
 \Pol_n:= 
 \left\{  t^n+\sum_{\ell=0}^{n-1}\alpha_\ell t^\ell: \alpha_0,\ldots,\alpha_{n-1}\in\C \right\}
\]
be the affine space of monic polynomials of degree $n$. We identify $\Pol_n$ with $\C^n$ through the coefficient vector $(\alpha_0,\ldots,\alpha_{n-1})$. For a nonzero polynomial $Q\in\C[t]$, we denote its zero set by
\[
 Z(Q):=\{\lambda\in\C:Q(\lambda)=0\}.
\]
Every finite-dimensional complex vector space is equipped with its unique Hausdorff vector topology. A map $F:V\to W$ between such spaces is \emph{polynomial} if its coordinate functions with respect to some (equivalently, any) bases of $V$ and $W$ are polynomials in the coordinates on $V$. We denote by $\Poly(V)$ the algebra of polynomial functions $V\to\C$.

\subsection{Minimal polynomials and spectra}\label{subsec:calculus}

We define the spectrum of an element through its minimal polynomial and show that it agrees with the usual spectrum for unital associative and Jordan algebras.

Set $\A^1:=\A$ if $\A$ is unital. Otherwise, let $\A^1:=\A\oplus\C 1$ be its \emph{(external) unitization}~\cite[p.~11]{Schafer1966}, with product
\[
 (x+\alpha1)(y+\beta 1):=xy+\alpha y+\beta x+\alpha\beta 1,
 \qquad x,y\in\A,\quad \alpha,\beta\in\C.
\]
For each $x\in\A$, adjoining $1$ to the associative subalgebra of $\A$ generated by $x$ gives an associative subalgebra of $\A^1$. Every element of $\A^1$ lies in one of these subalgebras, so $\A^1$ is power-associative.

For $x\in\A$, we put $x^0:=1\in\A^1$ and $x^{k+1}:=x^kx$ for $k\in\N_0$. By power-associativity, evaluation at $x$, with constants interpreted as scalar multiples of $1\in\A^1$, is a homomorphism of $\C[t]$ onto the unital commutative associative subalgebra
\[
 \C[x]:=\{f(x):f\in\C[t]\}\subseteq\A^1.
\]
Since $\A^1$ is finite-dimensional and $1\ne0$, the kernel is a nonzero proper ideal of $\C[t]$. Its unique monic generator $\mu_x\in \C[t]$ has positive degree and is called the \emph{minimal polynomial} of $x$~\cite[\S\,1]{Goldberg2007}. Evaluation induces an isomorphism
\[
 \C[t]/(\mu_x)\cong\C[x],
\]
where $(\mu_x)$ denotes the ideal generated by $\mu_x$. We define the \emph{spectrum} of $x\in\A$ by
\begin{equation}\label{eq:spc_powerassoc}
 \spc_\A(x):=\{\lambda\in\C:\mu_x(\lambda)=0\}.
\end{equation}
Every spectrum is therefore nonempty and finite. If $\A$ is nonunital, the scalar projection $\A^1\to\C$, $y+\alpha1\mapsto\alpha$ for $y\in\A$ and $\alpha\in\C$, is a unital homomorphism; applying it to $\mu_x(x)=0$ gives $\mu_x(0)=0$, so
\begin{equation}\label{eq:zero-in-spectrum}
 0\in\spc_\A(x),\qquad x\in\A,\qquad \text{whenever $\A$ is nonunital.}
\end{equation}
For $\lambda\in\C$, B\'ezout's identity shows that $\lambda 1-x$ is invertible in $\C[x]$ if and only if $\lambda-t$ and $\mu_x$ are coprime, or equivalently, $\mu_x(\lambda)\ne0$. Thus
\begin{equation}\label{eq:intrinsic-spectrum}
 \spc_\A(x)=\{\lambda\in\C:\lambda 1-x  \text{ is not invertible in }\C[x]\}.
\end{equation}

For a unital associative algebra $\A$, \eqref{eq:intrinsic-spectrum} gives the usual spectrum~\eqref{eq:spc_assoc}. Indeed, if $y\in\A$ is invertible, multiplication by $y$ is an injective linear map of $\C[y]$ into itself. Since $\C[y]$ is finite-dimensional, this map is surjective, so $y^{-1}\in\C[y]$. As $\C[\lambda 1-x]=\C[x]$ for every $x\in\A$ and
$\lambda\in\C$, the element $\lambda 1-x$ is invertible in $\A$
if and only if it is invertible in $\C[x]$.

For a unital Jordan algebra $\J$, the Linear Inverse Lemma~\cite[Part~II, 6.1.7]{McCrimmon2004} states that an element $z\in\J$ is \emph{Jordan invertible} if and only if there exists $y\in\J$ such that
\[
 z\circ y=1,\qquad z^2\circ y=z.
\]
In finite dimensions, Jordan invertibility is equivalent to invertibility in $\C[z]$. Indeed, an inverse in $\C[z]$ satisfies these two identities. Conversely, if $z$ is Jordan invertible, its \emph{quadratic operator} $U_z:\J\to\J$, defined by
\begin{equation}\label{eq:quadratic-operator}
 U_z(w):=2z\circ(z\circ w)-z^2\circ w,\qquad w\in\J,
\end{equation}
is bijective by the Invertibility Criterion~\cite[Part~II, 6.1.2]{McCrimmon2004}. It leaves $\C[z]$ invariant and acts there as multiplication by $z^2$. Its restriction is injective and hence surjective, since $\C[z]$ is finite-dimensional. Thus $z^2\circ v=1$ for some $v\in\C[z]$, and $z\circ v$ is an inverse of $z$ in $\C[z]$. Since $\C[\lambda 1-z]=\C[z]$,
\eqref{eq:intrinsic-spectrum} gives the usual
\emph{Jordan spectrum}:
\[
 \spc_\J(z)=\{\lambda\in\C:\lambda 1-z  \text{ is not Jordan invertible}\}.
\]

\begin{remark}\label{rem:spectrum-conventions}
For a unital Jordan algebra $\J$ and $x\in\J$, the Jordan spectrum $\spc_\J(x)$ need not coincide with the \emph{multiplicative spectrum}~\cite[Definition~2.1]{MarcosVelasco2014}, defined as the usual spectrum of the multiplication operator $L_x:\J\to\J$, where $L_x(y):=x\circ y$ for $y\in\J$. For example, take $\J:=M_2(\C)^+$ and define $X,Y\in\J$ by
\[
 X:=\begin{bmatrix}1&0\\0&3\end{bmatrix},\qquad
 Y:=\begin{bmatrix}0&1\\1&0\end{bmatrix}.
\]
Since $Y\ne0$ and $L_X(Y)=2Y$, the multiplicative spectrum of $X$ contains $2$, whereas $\spc_\J(X)=\{1,3\}$.
\end{remark}

\subsection{The semisimple Jordan quotient}\label{subsec:quotient}

We now construct the semisimple Jordan quotient used in the spectral reduction. The \emph{symmetrization} $\A^+$ is the vector space $\A$ equipped with the commutative product~\eqref{eq:symmetrized-product}. For $x\in\A$ and $k,\ell\in\N$, power-associativity gives $x^k\circ x^\ell=x^{k+\ell}$, so $\A^+$ is power-associative and has the same powers as $\A$.

The algebras $\A$ and $\A^+$ also have the same identity elements. An identity of $\A$ is an identity of $\A^+$. Conversely, suppose that $e\in\A$ is an identity of $\A^+$. Then $e^2=e$ and $ex+xe=2x$ for every $x\in\A$. For $x\in\A$ and $\lambda\in\C$, power-associativity gives
\[
 (e+\lambda x)^2(e+\lambda x)=(e+\lambda x)(e+\lambda x)^2.
\]
Comparing the coefficients of $\lambda$ yields $ex+2xe=2ex+xe$, so $ex=xe=x$. Consequently, $(\A^+)^1=(\A^1)^+$, and polynomial evaluation, minimal polynomials and spectra agree in $\A$ and $\A^+$.

Let $\phi:\A^+\to\B^+$ be a homomorphism. If $\A$ is unital, assume that $\B$ is unital and $\phi(1_\A)=1_\B$. The map $\phi$ preserves positive powers. Since $\mu_x(0)=0$ for $x\in\A$ when $\A$ is nonunital by~\eqref{eq:zero-in-spectrum}, in either case we have
\begin{equation}\label{eq:hom-shrinking}
 \mu_x(\phi(x))=\phi(\mu_x(x))=0,\qquad x\in\A.
\end{equation}
Hence $\mu_{\phi(x)}$ divides $\mu_x$, so $\spc_\B(\phi(x))\subseteq\spc_\A(x)$ for every $x\in\A$. In particular, unital Jordan homomorphisms are spectrum-shrinking.

Every finite-dimensional power-associative algebra has a largest nil ideal. Indeed, suppose first that $\mathcal I$ is a nil ideal of $\A$ and that $\A/\mathcal I$ is nil. For each $x\in\A$, there are $\ell,k\in\N$ such that $x^\ell\in\mathcal I$ and $(x^\ell)^k=0$. Power-associativity gives $x^{\ell k}=0$, so $\A$ is nil.

Since $\A$ is finite-dimensional, choose a nil ideal $\mathcal I$ of $\A$ of maximal dimension. Let $\mathcal I'$ be any nil ideal of $\A$. The quotient $(\mathcal I+\mathcal I')/\mathcal I$ is nil, being a homomorphic image of $\mathcal I'$. The preceding observation, applied to $\mathcal I+\mathcal I'$ and its ideal $\mathcal I$, shows that $\mathcal I+\mathcal I'$ is nil. Maximality of $\dim_\C\mathcal I$ therefore gives $\mathcal I+\mathcal I'=\mathcal I$, so $\mathcal I'\subseteq\mathcal I$. This largest nil ideal is called the \emph{nilradical} of $\A$ and is denoted by $\Nil(\A)$~\cite[p.~130]{Schafer1966}. 

By the same observation, the preimage in $\A$ of any nil ideal of $\A/\Nil(\A)$ is a nil ideal of $\A$ containing $\Nil(\A)$, and hence equals $\Nil(\A)$. Consequently, $\A/\Nil(\A)$ has zero nilradical.

A finite-dimensional complex Jordan algebra is \emph{semisimple} if its nilradical is zero~\cite[p.~96]{Schafer1966}. By a theorem of Albert~\cite[Theorem~8]{Albert1950}, every such algebra is a finite direct sum of simple unital Jordan algebras (see also~\cite[Part~I, \S\,2.13]{McCrimmon2004} and~\cite[pp.~98--99]{Schafer1966}).

Although $\A^+$ need not be a Jordan algebra~\cite[p.~504]{Albert1950}, Kokoris's theorem~\cite[Theorem]{Kokoris1956} states that every finite-dimensional commutative complex power-associative algebra with zero nilradical is a Jordan algebra (see also~\cite[p.~132]{Schafer1966}). Extending the associative reduction of~\cite{GT2025}, we set
\[
 \widehat\A:=\A^+/\Nil(\A^+) \qquad \text{and} \qquad  \widehat\B:=\B^+/\Nil(\B^+).
\]
Both quotients are commutative and power-associative and have zero nilradical, so they are semisimple Jordan algebras.

We denote the corresponding quotient maps by
\[
 \pi_\A:\A\to\widehat\A \qquad \text{and} \qquad  \pi_\B:\B\to\widehat\B.
\]
They are linear and preserve the symmetrized products and all positive powers. As symmetrization preserves powers, $\widehat\A=\{0\}$ if and only if $\A$ is nil. In particular, if $\A$ is unital, then $\widehat\A\ne\{0\}$ and $\pi_\A$ is unital. The same statements hold for $\B$.

We fix decompositions into simple Jordan algebras
\begin{equation}\label{eq:decompositions}
 \widehat\A\cong\bigoplus_{j=1}^{p}\J_j, \qquad
 \widehat\B\cong\bigoplus_{k=1}^{q}\K_k,
\end{equation}
where $p,q\in\N_0$ and an empty direct sum is the zero algebra. We identify $\widehat\A$ and $\widehat\B$ with the respective direct sums in~\eqref{eq:decompositions}. We write
\[
 \pi_{\A,j}:\A\to\J_j,\qquad 1\le j\le p 
 \qquad \text{and} \qquad 
 \pi_{\B,k}:\B\to\K_k,\qquad 1\le k\le q,
\]
for the coordinate maps of $\pi_\A$ and $\pi_\B$ under~\eqref{eq:decompositions}.

\subsection{Generic minimal polynomials}\label{subsec:generic}

We recall the generic minimal polynomial and establish the specialization and uniqueness properties needed for the polynomial sections.

Let $\J$ be a unital finite-dimensional complex Jordan algebra with identity $1_\J$. Over $\C$, the Jordan identity is equivalent to its multilinear polarization, so every scalar extension of $\J$ is again Jordan and hence power-associative. Thus $\J$ is \emph{strictly power-associative}~\cite[\S\,1, p.~26]{Jacobson1963}.

Put $d:=\dim_\C\J$ and fix a basis $\{e_1,\ldots,e_d\}$ of $\J$. Let $\xi_1,\ldots,\xi_d\in\Poly(\J)$ be the corresponding coordinate functions, so that $x=\sum_{s=1}^d\xi_s(x)e_s$ for $x\in\J$. Set $\mathbb K:=\operatorname{Frac}(\Poly(\J))$, the field of rational functions on $\J$. We pass to the scalar extension $\mathbb K\otimes_\C\J$, a finite-dimensional Jordan algebra over $\mathbb K$, and consider the \emph{generic element} of $\J$~\cite[\S\,1, p.~26]{Jacobson1963},
\[
 \mathbf x:=\sum_{s=1}^d\xi_s\otimes e_s
 \in\mathbb K\otimes_\C\J.
\]
Let $\mu_{\mathbf x}\in\mathbb K[t]$ be its minimal polynomial in this algebra and set $n:=\deg\mu_{\mathbf x}$. By Jacobson's construction~\cite[\S\,1, pp.~26--27]{Jacobson1963}, the coefficients of $\mu_{\mathbf x}$ belong to $\Poly(\J)$. For $0\le\ell<n$, let $\gamma_\ell\in\Poly(\J)$ denote the coefficient of $t^\ell$ in $\mu_{\mathbf x}$. Specializing these coefficients at $x\in\J$ gives the \emph{generic minimal polynomial} of $x$:
\[
 \chi_\J(x;t):=t^n+\sum_{\ell=0}^{n-1}\gamma_\ell(x)t^\ell.
\]
Thus $x\mapsto\chi_\J(x;\cdot)$ is a polynomial map $\J\to\Pol_n$.

By bilinearity, for each $k\in\N_0$ there are polynomial functions $p_{sk}\in\Poly(\J)$, $1\le s\le d$, such that
\[
 \mathbf x^k=\sum_{s=1}^d p_{sk}\otimes e_s, \qquad
 x^k=\sum_{s=1}^d p_{sk}(x)e_s, \qquad x\in\J.
\]
Expressing $\mu_{\mathbf x}(\mathbf x)=0$ in the basis $1\otimes e_1,\ldots,1\otimes e_d$ gives polynomial identities on $\J$. Evaluating these identities at $x\in\J$ shows that $\chi_\J(x;\cdot)$ annihilates $x$. Hence
\[
 \mu_x\mid\chi_\J(x;\cdot), \qquad \deg\mu_x\le n,
 \qquad x\in\J.
\]
Since $\deg\mu_{\mathbf x}=n$, the $d\times n$ matrix $[p_{sk}]$, where $1\le s\le d$ and $0\le k<n$, has rank $n$ over $\mathbb K$. For $x\in\J$, the elements $1_\J,x,\ldots,x^{n-1}$ are linearly independent if and only if $\deg\mu_x=n$, since $\deg\mu_x\le n$. Consequently, the set
\[
 \{x\in\J:\deg\mu_x=n\}
\]
is the union of the nonvanishing sets of the $n\times n$ minors of $[p_{sk}]$. It is open by continuity and dense because at least one of these minors is a nonzero polynomial function, whose zero set has empty interior in $\J$. We call $n$ the \emph{Jordan rank} (or \emph{degree}) of $\J$
and write
\[
 \rank\J:=n=\max\{\deg\mu_x:x\in\J\}.
\]
Moreover, the map $x\mapsto\chi_\J(x;\cdot)$ is the unique continuous map $\J\to\Pol_n$ whose value at each $x\in\J$ annihilates $x$. Indeed, let $F:\J\to\Pol_n$ be another such map. For $x\in\J$ with $\deg\mu_x=n$, the polynomials $F(x)$ and $\chi_\J(x;\cdot)$ are monic of degree $n$ and divisible by $\mu_x$, so both equal $\mu_x$. The two maps therefore agree on a dense subset of $\J$, and hence everywhere by continuity. In particular, $\chi_\J$ is independent of the chosen basis.

For every $x\in\J$, each irreducible factor of $\chi_\J(x;\cdot)$ divides $\mu_x$~\cite[Theorem~1(v)]{Jacobson1963}. Together with $\mu_x\mid\chi_\J(x;\cdot)$, this gives
\begin{equation}\label{eq:characteristic-spectrum}
 Z(\chi_\J(x;\cdot))=Z(\mu_x)
 \stackrel{\eqref{eq:spc_powerassoc}}{=}\spc_\J(x),
 \qquad x\in\J.
\end{equation}
For $\lambda\in\C$, the minimal polynomial of $\lambda 1_\J$ is $t-\lambda$. Hence $\chi_\J(\lambda 1_\J;\cdot)$ is monic of degree $n$ with $\lambda$ as its only root, so
\begin{equation}\label{eq:scalar-characteristic}
 \chi_\J(\lambda 1_\J;t)=(t-\lambda)^n,
 \qquad\lambda\in\C.
\end{equation}

\subsection{The simple Jordan algebras}\label{subsec:models}

We recall the standard models of the simple summands in~\eqref{eq:decompositions} and their Jordan ranks.

Let $V$ be a finite-dimensional complex vector space with $\dim_\C V\ge2$, and let $b:V\times V\to\C$ be a nondegenerate symmetric bilinear form. The \emph{spin factor} $\mathcal S(V,b):=\C 1\oplus V$
has identity $1$ and product
\begin{equation}\label{eq:spin-product}
 (\alpha1+u)\circ(\beta 1+v)
 :=(\alpha\beta+b(u,v))1+\alpha v+\beta u,
 \qquad \alpha,\beta\in\C,\quad u,v\in V.
\end{equation}
For $\alpha\in\C$ and $v\in V$, the polynomial $(t-\alpha)^2-b(v,v)$ annihilates $\alpha1+v$, so every minimal polynomial has degree at most $2$. By nondegeneracy, we can choose $v\in V$ with $b(v,v)=1$. Then $v^2=1$ and $v$ is not scalar, so $\mu_v=t^2-1$. Hence $\mathcal S(V,b)$ has Jordan rank $2$.

For $n\in\N$ and $A\in M_n(\C)$, we write $A^t$
for its transpose and
\[
 \chi(A;t):=\det(tI_n-A)
\]
for its characteristic polynomial. We use $\diag(\cdot)$ for diagonal and block diagonal matrices. The \emph{full matrix Jordan algebra} $M_n(\C)^+$ carries the product~\eqref{eq:symmetrized-product} and has identity $I_n$. The \emph{symmetric matrix algebra}
\[
 H_n(\C):=\{A\in M_n(\C):A^t=A\}
\]
is a Jordan subalgebra of $M_n(\C)^+$ with the same identity.

In both algebras, Jordan powers are ordinary matrix powers, so minimal polynomials are the usual matrix minimal polynomials. The Cayley--Hamilton theorem gives $\deg\mu_A\le n$, with equality at $A=\diag(1,2,\ldots,n)$. Thus both algebras have Jordan rank $n$. Since $A\mapsto\chi(A;\cdot)$ is a polynomial map into $\Pol_n$ whose value at $A$ annihilates $A$, the uniqueness property established in Section~\ref{subsec:generic} gives
\begin{equation}\label{eq:matrix-characteristic}
 \chi_{H_n(\C)}(A;t)=\chi(A;t), \qquad A\in H_n(\C),
\end{equation}
and likewise
$\chi_{M_n(\C)^+}(A;t)=\chi(A;t)$ for $A\in M_n(\C)$.

For $n\in\N$, set
\[
 R_n:=\begin{bmatrix}0&I_n\\-I_n&0\end{bmatrix} \in M_{2n}(\C).
\]
The \emph{symplectic Hermitian algebra}
\begin{equation}\label{eq:symplectic-model}
 H_{2n}(\C,\mathrm{sp}):=\{A\in M_{2n}(\C):A^tR_n=R_nA\}
\end{equation}
is a Jordan subalgebra of $M_{2n}(\C)^+$ with identity $I_{2n}$.

Finally, let $\mathbb O_s$ be the split octonion algebra over $\C$. Its standard conjugation $\mathbb O_s\to\mathbb O_s$, $z\mapsto\overline z$, is linear, and we identify $\C$ with $\C 1_{\mathbb O_s}$. The \emph{split Albert algebra} is
\begin{equation}\label{eq:albert-model}
 H_3(\mathbb O_s):=
 \left\{
 \begin{bmatrix}
  \alpha&a&b\\
  \overline a&\beta&c\\
  \overline b&\overline c&\gamma
 \end{bmatrix}
 :\alpha,\beta,\gamma\in\C,\, a,b,c\in\mathbb O_s
 \right\},
\end{equation}
with product $A\circ B:=\frac12(AB+BA)$ for $A,B\in H_3(\mathbb O_s)$, where matrix products are formed with octonion multiplication, and identity $I_3$. The algebra $\mathbb O_s$ is alternative but not associative~\cite[Part~I,  \S\S\,~2.10--2.12]{McCrimmon2004}. Moreover, $H_3(\mathbb O_s)$ is \emph{exceptional}: it is not isomorphic to a Jordan subalgebra of $\mathcal D^+$ for any associative complex algebra $\mathcal D$~\cite[Part~I,  \S\S\,~1.10 and~2.13]{McCrimmon2004}.

By~\cite[\S\,1, p.~27 and \S\,4, pp.~37--38]{Jacobson1963}, we have $\deg\mu_A\le n$ for $A\in H_{2n}(\C,\mathrm{sp})$ and $\deg\mu_A\le3$ for $A\in H_3(\mathbb O_s)$. Equality is attained at $\diag(D,D)\in H_{2n}(\C,\mathrm{sp})$, where $D:=\diag(1,\ldots,n)\in M_n(\C)$, and at $\diag(1,2,3)\in H_3(\mathbb O_s)$, respectively.

We use the following classification~\cite[Part~I, \S\,2.13]{McCrimmon2004}.

\begin{theorem}[Classification over $\C$]\label{thm:Jordan-models}
Up to isomorphism, the finite-dimensional simple complex Jordan algebras are precisely the following:
\begin{enumerate}[label=\textup{(\alph*)}]
\item the scalar algebra $\C$;
\item the spin factors $\mathcal S(V,b)$, where
$\dim_\C V\ge2$ and $b$ is nondegenerate and symmetric;
\item $M_n(\C)^+$, where $n\ge3$;
\item $H_n(\C)$, where $n\ge3$;
\item $H_{2n}(\C,\mathrm{sp})$, where $n\ge3$;
\item the split Albert algebra $H_3(\mathbb O_s)$.
\end{enumerate}
Their Jordan ranks in \textup{(a)}--\textup{(f)}
are $1$, $2$, $n$, $n$, $n$ and $3$, respectively.
\end{theorem}

The restrictions $n\ge3$ in~\textup{(c)}--\textup{(e)} avoid repetitions with~\textup{(a)} and~\textup{(b)}. For $n=1$, the three matrix models are isomorphic to $\C$, and for $n=2$, they are isomorphic to spin factors.

\subsection{Spectral coordinates}\label{subsec:coordinates}

We now express the spectrum in terms of the generic minimal polynomials of the simple quotient coordinates.

For the decompositions~\eqref{eq:decompositions}, set
\begin{equation}\label{eq:spectral-factors}
 \begin{alignedat}{3}
 n_j&:=\rank\J_j,\qquad
 &P_j(x;t)&:=\chi_{\J_j}(\pi_{\A,j}(x);t),\qquad
 &x&\in\A,\quad 1\le j\le p,\\
 m_k&:=\rank\K_k,\qquad
 &T_k(y;t)&:=\chi_{\K_k}(\pi_{\B,k}(y);t),\qquad
 &y&\in\B,\quad 1\le k\le q.
 \end{alignedat}
\end{equation}
The maps $P_j:\A\to\Pol_{n_j}$ and $T_k:\B\to\Pol_{m_k}$ are polynomial. Their values are called the \emph{source factors} and \emph{target factors}, respectively.

We set $j_\A:=1$ if $\A$ is unital and $j_\A:=0$ otherwise. In the nonunital case, we also define
\begin{equation}\label{eq:zero-factor}
 n_0:=1,\qquad P_0(x;t):=t,\qquad x\in\A.
\end{equation}
This factor accounts for the spectral value $0$ contributed by unitization. Thus the source indices range from $j_\A$ to $p$. Empty sums and products are $0$ and $1$, respectively.

\begin{lemma}\label{lem:reduction}
For every $x\in\A$,
\begin{equation}\label{eq:spectral-reduction}
 \spc_\A(x)=Z\left(\prod_{j=j_\A}^{p}P_j(x;\cdot)\right).
\end{equation}
\end{lemma}

\begin{proof}
Fix $x\in\A$ and define $Q\in\C[t]$ by
\[
 Q(t):=\prod_{j=j_\A}^{p}P_j(x;t).
\]
The polynomial $Q$ has zero constant term when $\A$ is nonunital; thus $Q(x)\in\A$ in either case. Each coordinate map $\pi_{\A,j}:\A^+\to\J_j$, $1\le j\le p$, is a homomorphism that is unital when $\A$ is unital, so it commutes with evaluation of $Q$. By~\eqref{eq:hom-shrinking}, $\mu_{\pi_{\A,j}(x)}$ divides $\mu_x$. Hence~\eqref{eq:characteristic-spectrum} gives
\[
 Z(P_j(x;\cdot))\subseteq\spc_\A(x).
\]
Together with~\eqref{eq:zero-in-spectrum} in the nonunital case, this yields $Z(Q)\subseteq\spc_\A(x)$.

Conversely, $P_j(x;\cdot)$ annihilates
$\pi_{\A,j}(x)$ and divides $Q$, so
\[
 \pi_{\A,j}(Q(x))=Q(\pi_{\A,j}(x))=0, \qquad 1\le j\le p.
\]
Consequently,
\[
 Q(x)\in\ker\pi_\A=\Nil(\A^+).
\]
Since symmetrization preserves powers, $Q(x)^r=0$ for some $r\in\N$. Evaluation in $\C[x]$ gives
$(Q^r)(x)=Q(x)^r=0$, so $\mu_x$ divides $Q^r$. Therefore $\spc_\A(x)\subseteq Z(Q)$, proving
\eqref{eq:spectral-reduction}.
\end{proof}

\section{Main result}\label{sec:main}

Our main theorem expresses the target factors of continuous spectrum-shrinking maps as products of source factors. It gives exact existence criteria and characterizes spectrum preservation.

\begin{theorem}\label{thm:classification}
Let $\A$ and $\B$ be finite-dimensional complex power-associative algebras, with the notation of Subsection~\ref{subsec:coordinates}.
\begin{enumerate}[label=\textup{(\alph*)}]
\item If $\A$ is unital and $\B$ is nonunital, there is no spectrum-shrinking map $\A\to\B$.

\item Suppose that $\A$ is nonunital or $\B$ is unital. A continuous map $\phi:\A\to\B$ is spectrum-shrinking if and only if there are integers $\ell_j^k\in\N_0$, $j_\A\le j\le p$, $1\le k\le q$, such that
\begin{equation}\label{eq:multiplicities}
 T_k(\phi(x);t)  =\prod_{j=j_\A}^{p}P_j(x;t)^{\ell_j^k},
 \qquad x\in\A,\quad 1\le k\le q.
\end{equation}
These exponents are uniquely determined by $\phi$ and satisfy the linear Diophantine system
\begin{equation}\label{eq:rank-equations}
 \sum_{j=j_\A}^{p}n_j\ell_j^k=m_k, \qquad 1\le k\le q.
\end{equation}
A continuous spectrum-shrinking map $\A\to\B$ exists if and only if~\eqref{eq:rank-equations} has a solution in nonnegative integers. Every such solution is realized by a polynomial spectrum-shrinking map.

\item A continuous spectrum-shrinking map $\phi:\A\to\B$, with exponents $\ell_j^k$ as in~\textup{(b)}, is spectrum-preserving if and only if
\begin{equation}\label{eq:coverage}
 \sum_{k=1}^{q}\ell_j^k>0, \qquad 1\le j\le p,
\end{equation}
and $\sum_{k=1}^{q}\ell_0^k>0$ when $\A$ is nonunital and $\B$ is unital. Consequently, a continuous spectrum-preserving map $\A\to\B$ exists if and only if the case in~\textup{(a)} does not occur and~\eqref{eq:rank-equations} has a solution in nonnegative integers satisfying these positivity conditions. Every such solution is realized by a polynomial spectrum-preserving map.
\end{enumerate}
\end{theorem}

For a continuous spectrum-shrinking map $\phi:\A\to\B$, \eqref{eq:multiplicities} shows that the maps $x\mapsto T_k(\phi(x);\cdot)$ from $\A$ to $\Pol_{m_k}$, $1\le k\le q$, are polynomial and factor through $\pi_\A$. The map $\phi$ itself need not be polynomial or factor through $\pi_\A$, as Example~\ref{ex:nil-dependence} shows.

In fact, parts \textup{(b)} and \textup{(c)} remain valid if continuity of $\phi$ is replaced by continuity of each of these maps, since the proof applies Lemma~\ref{lem:multiplicities} only to these maps. This assumption is strictly weaker in general. For example, the spectrum-preserving map $\psi:\C\to M_2(\C)$ defined by $\psi(0):=0$ and
\[
 \psi(\lambda):=
 \begin{bmatrix}\lambda&1\\0&\lambda\end{bmatrix},
 \qquad \lambda\in\C\setminus\{0\},
\]
is discontinuous at $0$, whereas its target-factor map $\lambda\mapsto(t-\lambda)^2$, $\lambda\in\C$, is polynomial.

Suppose now that $\A$ and $\B$ are unital and $\widehat\A$ and $\widehat\B$ are simple, of Jordan ranks $n$ and $m$, respectively. Then~\eqref{eq:rank-equations} reduces to $n\ell=m$ with $\ell\in\N_0$. Hence a continuous spectrum-shrinking map $\A\to\B$ exists if and only if $n$ divides $m$, and every such map is spectrum-preserving since $\ell=m/n>0$. For $\A=M_n(\C)$ and $\B=M_m(\C)$, the source and target factors are the ordinary matrix characteristic polynomials by Section~\ref{subsec:models}. Thus~\eqref{eq:multiplicities} recovers the matrix identity~\eqref{eq:matrix-multiplicities} of~\cite[Corollary~2.2(a)]{CGT2026}.

\subsection{Polynomial sections}\label{subsec:sections}

For each simple complex Jordan algebra $\J$, we construct a polynomial right inverse of the map
$\J\to\Pol_{\rank\J}$, $x\mapsto\chi_\J(x;\cdot)$. These sections produce elements with prescribed generic minimal polynomials without ordering their roots and give the polynomial realizations in Theorem~\ref{thm:classification}.

Fiedler constructed symmetric matrices with a prescribed characteristic polynomial~\cite[Theorem~2]{Fiedler1990}. For $n\ge2$, a polynomial right inverse of $A\mapsto\chi(A;\cdot)$ on $H_n(\C)$ also follows from~\cite[Theorem~12]{KostantRallis1971}, applied to traceless symmetric matrices and then adding the scalar part. We give an elementary construction using Toeplitz matrices, whose entries are constant along each diagonal. The construction combines the coefficient formula of~\cite[Lemma~4]{MMP1999} with the fixed similarity of~\cite[Theorem~3.1]{CLNT2017}; we include the details needed to obtain polynomial dependence on the prescribed coefficients.

\begin{proposition}\label{prop:sections}
Let $\J$ be a finite-dimensional simple complex Jordan algebra of Jordan rank $n\in\N$. There is a polynomial map $\Gamma_\J:\Pol_n\to\J$ such that
\begin{equation}\label{eq:section}
 \chi_\J(\Gamma_\J(Q);t)=Q(t), \qquad Q\in\Pol_n.
\end{equation}
\end{proposition}

\begin{proof}
\emph{Step~1: symmetric matrices.}
Let $N\in M_n(\C)$ have ones on the superdiagonal and zeros elsewhere. For $\beta_0,\ldots,\beta_{n-1}\in\C$, set $v:=(\beta_0,\ldots,\beta_{n-1})\in\C^n$. Define a polynomial map
\[
 M:\C^n\to M_n(\C), \qquad
 M(v):=N+\sum_{r=0}^{n-1}\beta_r(N^t)^r.
\]
For $1\le k\le n$, define a map
\[
 c_k:\C^n\to\C, \qquad
 c_k(v):=(-1)^k
 \sum_{\substack{I\subseteq\{1,\ldots,n\}\\ |I|=k}}
 \det M(v)_{I,I},
\]
where $M(v)_{I,I}$ denotes the principal submatrix indexed by $I$. Then
\[
 \chi(M(v);t)=t^n+\sum_{k=1}^{n}c_k(v)t^{n-k}.
\]
We claim that there are polynomial maps $f_k:\C^{k-1}\to\C$, $2\le k\le n$, such that
\begin{equation}\label{eq:Toeplitz-coefficients}
 \begin{aligned}
 c_1(v)&=-n\beta_0,\\
 c_k(v)&=-(n-k+1)\beta_{k-1}
 +f_k(\beta_0,\ldots,\beta_{k-2}),
 \qquad2\le k\le n.
 \end{aligned}
\end{equation}
The first identity follows from the diagonal entries of $M(v)$. For $2\le k\le n$, consider a principal submatrix of order $k$ and a term in its determinant expansion that is not identically zero in the parameters. Each selected index occurs once as a row index and once as a column index, so the column-minus-row differences sum to zero. A factor $\beta_r$ contributes $-r$, whereas every difference is at most $1$, with equality exactly for superdiagonal entries. Since the remaining $k-1$ factors contribute at most $k-1$, it follows that $r\le k-1$. If $r=k-1$, the other $k-1$ factors are superdiagonal ones. The largest selected row cannot use a superdiagonal entry, so it contributes $\beta_{k-1}$, while every other selected row uses its superdiagonal entry. The selected indices are therefore consecutive, and the corresponding permutation sends the largest index to the smallest, forming a $k$-cycle. The cycle sign $(-1)^{k-1}$ and the factor $(-1)^k$ in $c_k$ give $-\beta_{k-1}$ for each of the $n-k+1$ intervals. All remaining terms involve only $\beta_0,\ldots,\beta_{k-2}$, which proves~\eqref{eq:Toeplitz-coefficients}.

Given $\alpha_0,\ldots,\alpha_{n-1}\in\C$, set
\[
 Q(t):=t^n+\sum_{j=0}^{n-1}\alpha_jt^j\in\Pol_n.
\]
Define $\beta_0,\ldots,\beta_{n-1}\in\C$ recursively by
\begin{align*}
 \beta_0&:=-\frac{\alpha_{n-1}}{n},\\
 \beta_{k-1}&:=
 \frac{f_k(\beta_0,\ldots,\beta_{k-2})-\alpha_{n-k}}
      {n-k+1},
 \qquad2\le k\le n.
\end{align*}
The recursion defines a polynomial map $\eta:\Pol_n\to\C^n$ by
\[
 \eta(Q):=(\beta_0,\ldots,\beta_{n-1}),\qquad Q\in\Pol_n.
\]
Thus $c_k(\eta(Q))=\alpha_{n-k}$ for $1\le k\le n$, so
\begin{equation}\label{eq:Toeplitz-characteristic}
 \chi(M(\eta(Q));t)=Q(t),
 \qquad Q\in\Pol_n.
\end{equation}

Let $E\in M_n(\C)$ be the reversal matrix, with ones on the antidiagonal and zeros elsewhere.
The identities $N^tE=EN$ and $NE=EN^t$ give
\begin{equation}\label{eq:Toeplitz-intertwining}
 M(v)^tE=EM(v),\qquad v\in\C^n.
\end{equation}
Define $D\in M_n(\C)$ by
\[
 D:=\frac12(I_n+E)+\frac i2(I_n-E),
\]
where $i\in\C$ denotes the imaginary unit. Since $E^t=E$ and $E^2=I_n$, we have
\begin{equation}\label{eq:Toeplitz-Gram}
 D^tD=E.
\end{equation}
In particular, $D$ is invertible. Define a map $\Gamma_n:\Pol_n\to M_n(\C)$ by
\[
 \Gamma_n(Q):=DM(\eta(Q))D^{-1},\qquad Q\in\Pol_n.
\]
Then
\[
\Gamma_n(Q)^t \stackrel{\eqref{eq:Toeplitz-Gram}}{=}
 (D^{-1})^tM(\eta(Q))^tED^{-1} \stackrel{\eqref{eq:Toeplitz-intertwining},\,
 \eqref{eq:Toeplitz-Gram}}{=}
 DM(\eta(Q))D^{-1}=\Gamma_n(Q).
\]
 Thus $\Gamma_n$ is polynomial and takes values in $H_n(\C)$. Similarity gives
\begin{equation}\label{eq:symmetric-section}
 \chi(\Gamma_n(Q);t)=\chi(M(\eta(Q));t)
 \stackrel{\eqref{eq:Toeplitz-characteristic}}{=}Q(t),
 \qquad Q\in\Pol_n.
\end{equation}

\emph{Step~2: the simple Jordan models.}
By Theorem~\ref{thm:Jordan-models}, it suffices to construct, for each standard model $\J$ of Jordan rank $n$, a unital injective Jordan homomorphism $h:H_n(\C)\to\J$. For $\C=H_1(\C)$ and $H_n(\C)$, take the identity map, and for $M_n(\C)^+$, take the natural inclusion.

For the symplectic Hermitian model~\eqref{eq:symplectic-model}, set $h(B):=\diag(B,B)$ for $B\in H_n(\C)$. Block multiplication gives
\[
 h(B)^tR_n=R_nh(B),\qquad  h(B)\circ h(B')
 \stackrel{\eqref{eq:symmetrized-product}}{=}h(B\circ B'), \qquad B,B'\in H_n(\C).
\]
This map is injective and sends $I_n$ to $I_{2n}$.

For a spin factor $\mathcal S(V,b)$, choose $u_1,u_2\in V$ such that
\[
 b(u_1,u_1)=b(u_2,u_2)=1,\qquad b(u_1,u_2)=0,
\]
which is possible because $b$ is nondegenerate over $\C$ and $\dim_\C V\ge2$. Define
\[
 h\left(
 \begin{bmatrix}
  \alpha+\beta&\gamma\\
  \gamma&\alpha-\beta
 \end{bmatrix}
 \right):=\alpha1+\beta u_1+\gamma u_2,
 \qquad\alpha,\beta,\gamma\in\C.
\]
The basis
\[
 I_2=\begin{bmatrix}1&0\\0&1\end{bmatrix},\qquad
 \begin{bmatrix}1&0\\0&-1\end{bmatrix},\qquad
 \begin{bmatrix}0&1\\1&0\end{bmatrix}
\]
of $H_2(\C)$ is sent to the linearly independent elements $1,u_1,u_2$. The two nonscalar basis matrices square to $I_2$ and have zero Jordan product. By~\eqref{eq:spin-product}, their images satisfy $u_1^2=u_2^2=1$ and $u_1\circ u_2=0$. Thus $h$ is a unital injective Jordan homomorphism.

For the split Albert algebra~\eqref{eq:albert-model}, take the inclusion $H_3(\C)\hookrightarrow H_3(\mathbb O_s)$. It is a unital injective Jordan homomorphism because standard conjugation fixes $\C$ and octonion multiplication restricts to complex multiplication. For an arbitrary simple $\J$, compose the corresponding model embedding with a Jordan isomorphism from that model onto $\J$.

For any such embedding $h$, the map
\[
 B\mapsto\chi_\J(h(B);\cdot), \qquad B\in H_n(\C),
\]
is polynomial with values in $\Pol_n$. Its value at $B$ annihilates $B$: indeed, it annihilates $h(B)$, and $h$ is injective and preserves polynomial evaluation. The uniqueness property established in Section~\ref{subsec:generic} therefore gives
\begin{equation}\label{eq:embedding-characteristic}
 \chi_\J(h(B);t)=\chi_{H_n(\C)}(B;t)
 \stackrel{\eqref{eq:matrix-characteristic}}{=}\chi(B;t),
 \qquad B\in H_n(\C).
\end{equation}
Setting $\Gamma_\J:=h\circ\Gamma_n$, we obtain
\[
 \chi_\J(\Gamma_\J(Q);t)
 \stackrel{\eqref{eq:embedding-characteristic}}{=}
 \chi(\Gamma_n(Q);t)
 \stackrel{\eqref{eq:symmetric-section}}{=}Q(t),
 \qquad Q\in\Pol_n,
\]
which proves~\eqref{eq:section}.
\end{proof}

\subsection{Multiplicity formula and proof of the main theorem}\label{subsec:proof}

We prove the multiplicity formula by combining the polynomial sections of Proposition~\ref{prop:sections} with the permutation argument of~\cite[proof of Theorem~2.1]{CGT2026}. We first record two elementary topological facts.

\begin{lemma}\label{lem:topology}
\begin{enumerate}[label=\textup{(\alph*)}]
\item Let $V$ be a finite-dimensional complex vector space and let $f\in\Poly(V)$ be nonzero. Then $\{v\in V:f(v)\ne0\}$ is open, dense and path-connected.

\item Let $\mathscr X$ be a nonempty connected topological space, $\mathscr Y$ a Hausdorff space, $r\in\N$, and $g,g_1,\ldots,g_r:\mathscr X\to\mathscr Y$ continuous maps such that, for every $w\in\mathscr X$,
\[
 g(w)\in\{g_1(w),\ldots,g_r(w)\},
 \qquad
 g_j(w)\ne g_k(w), \qquad 1\le j<k\le r.
\]
Then $g=g_s$ for some $1\le s\le r$~\cite[Lemma~2.7]{CGT2026}.
\end{enumerate}
\end{lemma}

\begin{proof}
\textup{(a)} The set is open by continuity of $f$, and it is dense because a nonzero polynomial function cannot vanish on a nonempty open set. Let $v,w\in V$ be distinct with $f(v)f(w)\ne0$. The polynomial $\lambda\mapsto f(v+\lambda(w-v))$, $\lambda\in\C$, is nonzero at $0$ and $1$ and has only finitely many zeros. A path from $0$ to $1$ in $\C$ avoiding these zeros yields a path from $v$ to $w$ in $\{x\in V:f(x)\ne0\}$.

\textup{(b)} The sets
$\{w\in\mathscr X:g(w)=g_s(w)\}$, $1\le s\le r$, are closed because $\mathscr Y$ is Hausdorff, and they are pairwise disjoint and cover $\mathscr X$. Hence each of them is also open, as the complement of the union of the others, and connectedness of $\mathscr X$ yields the claim.
\end{proof}

\begin{lemma}\label{lem:multiplicities}
Let $\A$ be a finite-dimensional complex power-associative algebra, with the notation of Subsection~\ref{subsec:coordinates}. Let $m\in\N$, and let $F:\A\to\Pol_m$ be continuous, with $Z(F(x))\subseteq\spc_\A(x)$ for every $x\in\A$. Write $F(x;t):=F(x)(t)$. Then there are unique integers $\ell_j\in\N_0$, $j_\A\le j\le p$, such that
\begin{equation}\label{eq:polynomial-multiplicities}
 F(x;t)=\prod_{j=j_\A}^{p}P_j(x;t)^{\ell_j}, \qquad x\in\A,
\end{equation}
and $\sum_{j=j_\A}^{p}n_j\ell_j=m$. In particular, $F$ is polynomial.
\end{lemma}

\begin{proof}
If $p=0$, then $\A$ is nonunital and~\eqref{eq:spectral-reduction} gives $\spc_\A(x)=\{0\}$ for every $x\in\A$. Hence $F(x;t)=t^m$, and the claim holds with $\ell_0=m$. Assume that $p\ge1$.

\emph{Step~1: determining the exponents.}
Let
\[
 W:=\bigoplus_{j=1}^{p}\C^{n_j},
\]
and write $z=(\zeta_{j,r})\in W$, with $\zeta_{j,r}\in\C$ for $1\le j\le p$ and $1\le r\le n_j$. Put $\zeta_{0,1}:=0$ when $\A$ is nonunital. Define
\[
 \mathscr E:=\left\{z\in W:
 \zeta_{j,r}\ne\zeta_{k,s}
 \text{ for all distinct }(j,r),(k,s)\right\},
\]
where $j_\A\le j,k\le p$, $1\le r\le n_j$ and $1\le s\le n_k$. This is the nonvanishing set of a product of nonzero linear polynomials: the differences between distinct coordinate functions on $W$ and, when $\A$ is nonunital, the coordinate functions on $W$. Thus $\mathscr E$ is nonempty and connected by Lemma~\ref{lem:topology}\textup{(a)}.

For $j_\A\le j\le p$, define $Q_j:W\to\Pol_{n_j}$ by
\[
 Q_j(z;t):=\prod_{r=1}^{n_j}(t-\zeta_{j,r}),\qquad z\in W.
\]
For $1\le j\le p$, let $\Gamma_{\J_j}:\Pol_{n_j}\to\J_j$ be a polynomial section given by Proposition~\ref{prop:sections}. Using~\eqref{eq:decompositions}, choose a linear right inverse $s_\A:\bigoplus_{j=1}^{p}\J_j\to\A$ of $\pi_\A$, and define $\Theta:W\to\A$ by
\[
 \Theta(z):=
 s_\A\left(\Gamma_{\J_1}(Q_1(z)),\ldots,\Gamma_{\J_p}(Q_p(z))\right),
 \qquad z\in W.
\]
The maps $Q_j$, $j_\A\le j\le p$, and $\Theta$ are polynomial. Since $\pi_\A\circ s_\A$ is the identity, we have
\begin{equation}\label{eq:source-section}
 P_j(\Theta(z);t)
 \stackrel{\eqref{eq:spectral-factors},\,\eqref{eq:section}}{=}Q_j(z;t),
 \qquad z\in W,\quad 1\le j\le p.
\end{equation}
When $\A$ is nonunital, \eqref{eq:source-section} also holds for $j=0$, since both sides are the polynomial $t$.

Consider the finite set
\[
 \Lambda_m:=
 \left\{
 \boldsymbol a=(a_{j,r})\in\prod_{j=j_\A}^{p}\N_0^{n_j}:
 \sum_{j=j_\A}^{p}\sum_{r=1}^{n_j}a_{j,r}=m
 \right\}.
\]
For $\boldsymbol a\in\Lambda_m$, define $G_{\boldsymbol a}:\mathscr E\to\Pol_m$ by
\[
 G_{\boldsymbol a}(z;t):=
 \prod_{j=j_\A}^{p}\prod_{r=1}^{n_j}(t-\zeta_{j,r})^{a_{j,r}},
 \qquad z\in \mathscr E.
\]
These maps are continuous, and their values at each point of $\mathscr E$ are pairwise distinct. By~\eqref{eq:source-section}, the spectral reduction~\eqref{eq:spectral-reduction} and the hypothesis on $F$, the polynomial $F(\Theta(z))$ is one of these values for every $z\in \mathscr E$. Lemma~\ref{lem:topology}\textup{(b)} therefore gives a single array $\boldsymbol a\in\Lambda_m$ such that $F\circ\Theta=G_{\boldsymbol a}$ on $\mathscr E$.

Fix $z\in \mathscr E$, $j\in\{1,\ldots,p\}$ and $r,s\in\{1,\ldots,n_j\}$ with $r<s$. Let $z'\in\mathscr E$ be obtained from $z$ by interchanging $\zeta_{j,r}$ and $\zeta_{j,s}$. The tuple $(Q_1(z),\ldots,Q_p(z))$ is unchanged, so $\Theta(z')=\Theta(z)$ and $G_{\boldsymbol a}(z')=G_{\boldsymbol a}(z)$. Comparing the multiplicity of $\zeta_{j,r}$ in these two polynomials gives $a_{j,r}=a_{j,s}$. Set $\ell_j:=a_{j,1}$ for $1\le j\le p$, and set $\ell_0:=a_{0,1}$ when $\A$ is nonunital. Then
\[
 \sum_{j=j_\A}^{p}n_j\ell_j=m.
\]
Consequently, the formula
\begin{equation}\label{eq:candidate-multiplicities}
 H(x;t):=\prod_{j=j_\A}^{p}P_j(x;t)^{\ell_j},
 \qquad x\in\A,
\end{equation}
defines a polynomial map $H:\A\to\Pol_m$, and
\begin{equation}\label{eq:model-multiplicities}
 F(\Theta(z);t)=\prod_{j=j_\A}^{p}Q_j(z;t)^{\ell_j}
 \stackrel{\eqref{eq:source-section},\,\eqref{eq:candidate-multiplicities}}{=}
 H(\Theta(z);t),\qquad z\in \mathscr E.
\end{equation}

\emph{Step~2: extending the formula.}
Set $d:=\sum_{j=j_\A}^{p}n_j\in\N$, and define a polynomial map $S:\A\to\Pol_d$ by
\[
 S(x;t):=\prod_{j=j_\A}^{p}P_j(x;t),\qquad x\in\A.
\]
For $x\in\A$, let $\lambda_1,\ldots,\lambda_d\in\C$ be the roots of $S(x;\cdot)$, listed with multiplicity. Its discriminant is
\[
 \disc S(x;\cdot):=\prod_{1\le r<s\le d}(\lambda_r-\lambda_s)^2.
\]
This symmetric expression is polynomial in the coefficients of $S(x;\cdot)$ and is nonzero exactly when $S(x;\cdot)$ has no repeated roots. Set
\[
 \mathscr U:=\{x\in\A:\disc S(x;\cdot)\ne0\}.
\]
By~\eqref{eq:source-section}, $\Theta(\mathscr E)\subseteq \mathscr U$, so the polynomial function $x\mapsto\disc S(x;\cdot)$ on $\A$ is not identically zero. Lemma~\ref{lem:topology}\textup{(a)} shows that $\mathscr U$ is open, dense and connected.

Consider
\[
 \mathscr C:=\{x\in \mathscr U:F(x)=H(x)\}.
\]
This set is nonempty by~\eqref{eq:model-multiplicities} and closed in $\mathscr U$ by continuity.
To prove that $\mathscr C$ is open in $\mathscr U$, fix $x_0\in \mathscr C$. For each $\lambda\in Z(S(x_0;\cdot))$, choose a closed disk $D_\lambda\subseteq\C$ of positive radius centered at $\lambda$, with the disks pairwise disjoint. Since $x_0\in \mathscr U$, each disk contains exactly one simple root of $S(x_0;\cdot)$. The roots of $F(x_0)$ and $H(x_0)$ belong to this same set, so none of these three polynomials vanishes on the boundary circles. Coefficient continuity and compactness of these circles give an open neighborhood $\mathscr V\subseteq \mathscr U$ of $x_0$ such that, for every $R\in\{S,F,H\}$ and every $y\in \mathscr V$,
\[
 |R(y;\zeta)-R(x_0;\zeta)|<|R(x_0;\zeta)|,
 \qquad \zeta\in\partial D_\lambda,\quad
 \lambda\in Z(S(x_0;\cdot)).
\]
Rouch\'e's theorem on each boundary circle shows that, for every $y\in \mathscr V$, each disk contains exactly one simple root of $S(y;\cdot)$, and the root counts of $F(y)$ and $H(y)$ in each disk equal their respective counts at $x_0$, with multiplicities. These counts agree because $F(x_0)=H(x_0)$. Since the number of disks equals the degree of $S(y;\cdot)$,
all its roots lie in the disks. Every root of $F(y)$ and $H(y)$ is a root of $S(y;\cdot)$, by the hypothesis on $F$,~\eqref{eq:spectral-reduction} and~\eqref{eq:candidate-multiplicities}. Thus $F(y)$ and $H(y)$ have the same roots with the same multiplicities, and monicity gives $F(y)=H(y)$. Therefore $\mathscr V\subseteq \mathscr C$, proving openness.

Connectedness gives $\mathscr C=\mathscr U$, and density and continuity yield $F=H$ on all of $\A$. Thus~\eqref{eq:polynomial-multiplicities} holds and $F$ is polynomial. At any point of $\mathscr U$, the source factors have pairwise disjoint simple roots, so comparison of root multiplicities proves uniqueness of the exponents.
\end{proof}

\begin{proof}[Proof of Theorem~\ref{thm:classification}]
\textup{(a)} Suppose that $\A$ is unital and $\B$ is nonunital. Then $\spc_\A(1_\A)=\{1\}$, whereas $0\in\spc_\B(y)$ for every $y\in\B$ by~\eqref{eq:zero-in-spectrum}. Hence no map $\A\to\B$ is spectrum-shrinking.

\textup{(b)} Assume that $\A$ is nonunital or $\B$ is unital, and let $\phi:\A\to\B$ be continuous and spectrum-shrinking. For $1\le k\le q$, the map $x\mapsto T_k(\phi(x);\cdot)$ from $\A$ to $\Pol_{m_k}$ is continuous. By Lemma~\ref{lem:reduction} and the hypothesis on $\phi$,
\[
 Z(T_k(\phi(x);\cdot)) \subseteq\spc_\B(\phi(x)) \subseteq\spc_\A(x),
 \qquad x\in\A.
\]
Applying Lemma~\ref{lem:multiplicities} for each $k$ gives~\eqref{eq:multiplicities} with uniquely determined exponents $\ell_j^k\in\N_0$. Comparing degrees gives~\eqref{eq:rank-equations}.

Conversely, suppose that a continuous map $\phi:\A\to\B$ satisfies~\eqref{eq:multiplicities} with nonnegative integer exponents. Fix $x\in\A$. By~\eqref{eq:multiplicities} and~\eqref{eq:spectral-reduction}, every root of each $T_k(\phi(x);\cdot)$ belongs to $\spc_\A(x)$. If $\B$ is nonunital, then $\A$ is also nonunital by assumption, so $0\in\spc_\A(x)$. Lemma~\ref{lem:reduction}, applied to $\B$, therefore gives $\spc_\B(\phi(x))\subseteq\spc_\A(x)$. Thus $\phi$ is spectrum-shrinking.

Now let $\ell_j^k\in\N_0$ satisfy~\eqref{eq:rank-equations} and put
\[
 Q_k(x;t):=\prod_{j=j_\A}^{p}P_j(x;t)^{\ell_j^k},  \qquad x\in\A,\quad 1\le k\le q.
\]
By~\eqref{eq:rank-equations}, these formulas define polynomial maps $Q_k:\A\to\Pol_{m_k}$. For $1\le k\le q$, let $\Gamma_{\K_k}:\Pol_{m_k}\to\K_k$ be a polynomial section given by Proposition~\ref{prop:sections}. Choose a linear right inverse $s_\B:\bigoplus_{k=1}^{q}\K_k\to\B$ of $\pi_\B$ and define $\phi:\A\to\B$ by
\[
 \phi(x):= s_\B\left(
 \Gamma_{\K_1}(Q_1(x;\cdot)),\ldots, \Gamma_{\K_q}(Q_q(x;\cdot))
 \right), \qquad x\in\A,
\]
with $\phi:=0$ when $q=0$. Then $\phi$ is polynomial, and
\[
 T_k(\phi(x);t)
 \stackrel{\eqref{eq:spectral-factors}}{=}
 \chi_{\K_k}\left(
 \Gamma_{\K_k}(Q_k(x;\cdot));t
 \right)
 \stackrel{\eqref{eq:section}}{=}Q_k(x;t),
 \qquad x\in\A,\quad 1\le k\le q.
\]
Thus $\phi$ satisfies~\eqref{eq:multiplicities} and is spectrum-shrinking by the converse proved above. Together with the necessity of~\eqref{eq:rank-equations}, this establishes the existence criterion and polynomial realization in~\textup{(b)}.

\textup{(c)} Let $\phi:\A\to\B$ be continuous and spectrum-shrinking, with exponents $\ell_j^k$ as in~\textup{(b)}. Suppose that~\eqref{eq:coverage} holds and that $\sum_{k=1}^{q}\ell_0^k>0$ when $\A$ is nonunital and $\B$ is unital. For every $x\in\A$, \eqref{eq:multiplicities} and~\eqref{eq:coverage} give
\[
 \bigcup_{j=1}^{p}Z(P_j(x;\cdot))
 \subseteq
 \bigcup_{k=1}^{q}Z(T_k(\phi(x);\cdot)).
\]
If $\A$ is nonunital, then $0\in\spc_\B(\phi(x))$ as well, since either $\B$ is nonunital or the additional positivity condition makes some target factor divisible by $t$. Applying Lemma~\ref{lem:reduction} to $\A$ and $\B$ therefore gives $\spc_\A(x)\subseteq\spc_\B(\phi(x))$. The reverse inclusion holds because $\phi$ is spectrum-shrinking, so $\phi$ is spectrum-preserving.

Conversely, suppose that $\phi$ is spectrum-preserving. Choose pairwise distinct nonzero scalars $\lambda_j\in\C$, $1\le j\le p$, and use the surjectivity of $\pi_\A$ to choose $x\in\A$ such that
\[
 \pi_{\A,j}(x)=\lambda_j1_{\J_j},
 \qquad 1\le j\le p.
\]
Then
\[
 P_j(x;t)
 \stackrel{\eqref{eq:spectral-factors},\,
 \eqref{eq:scalar-characteristic}}{=}
 (t-\lambda_j)^{n_j},
 \qquad 1\le j\le p.
\]
Since the $\lambda_j$ are distinct and nonzero, \eqref{eq:multiplicities} and Lemma~\ref{lem:reduction} give
\[
 \lambda_j\in\spc_\B(\phi(x))
 \quad\iff\quad
 \sum_{k=1}^{q}\ell_j^k>0,
 \qquad 1\le j\le p.
\]
Each $\lambda_j$ belongs to $\spc_\A(x)$, so spectrum preservation forces~\eqref{eq:coverage}. If $\A$ is nonunital and $\B$ is unital, the same formulas give
\[
 0\in\spc_\B(\phi(x))
 \quad\iff\quad
 \sum_{k=1}^{q}\ell_0^k>0.
\]
Since $0\in\spc_\A(x)$, the additional positivity condition follows.

The existence criterion and polynomial realization in~\textup{(c)} now follow from~\textup{(a)}, \textup{(b)} and the equivalence just proved.
\end{proof}

\section{Consequences and examples}\label{sec:consequences}

We now apply Theorem~\ref{thm:classification} to automatic spectrum preservation and scalar spectral selection, specialize the results to associative algebras, and conclude with examples illustrating the roles of symmetrization, continuity and the nilradical.

\subsection{Automatic spectrum preservation}\label{subsec:automatic}

We first determine when every continuous spectrum-shrinking map is automatically spectrum-preserving, and when spectrum-preserving maps exist from nonunital sources.

\begin{corollary}\label{cor:automatic}
Let $\A$ and $\B$ be unital finite-dimensional complex power-associative algebras, with the Jordan ranks in~\eqref{eq:spectral-factors}. Then every continuous spectrum-shrinking map $\A\to\B$ is spectrum-preserving if and only if for every $1\le j\le p$ there is $1\le k\le q$ such that
\begin{equation}\label{eq:omission-criterion}
 m_k\notin\left\{\sum_{\substack{1\le j'\le p\\j'\ne j}}n_{j'}\ell_{j'}:\ell_{j'}\in\N_0\text{ for }1\le j'\le p,\ j'\ne j\right\}.
\end{equation}
\end{corollary}

\begin{proof}
By Theorem~\ref{thm:classification}\textup{(b)} and~\textup{(c)}, including its polynomial realization, every continuous spectrum-shrinking map $\A\to\B$ is spectrum-preserving if and only if every nonnegative integer solution of~\eqref{eq:rank-equations} satisfies~\eqref{eq:coverage}. For fixed $1 \le j \le p$, a solution array omitting the $j$th source summand exists exactly when $m_k$ belongs to the set in~\eqref{eq:omission-criterion} for every $1\le k\le q$, since the equations for different target summands involve disjoint sets of unknowns. Negating the existence of such an array for some $j$ gives the stated criterion.
\end{proof}

Corollary~\ref{cor:automatic} includes the case in which no continuous spectrum-shrinking map exists: both its universal assertion and its rank criterion then hold. If $\A$ is unital and $\B$ is nonunital, there are no spectrum-shrinking maps $\A\to\B$ by Theorem~\ref{thm:classification}\textup{(a)}. Nonunital sources are treated next.

\begin{corollary}\label{cor:nonunital}
Let $\A$ and $\B$ be finite-dimensional complex power-associative algebras, with the Jordan ranks in~\eqref{eq:spectral-factors}, and suppose that $\A$ is nonunital.
\begin{enumerate}[label=\textup{(\alph*)}]
\item A continuous spectrum-preserving map $\A\to\B$ exists if and only if there is a map $\kappa:\{1,\ldots,p\}\to\{1,\ldots,q\}$ such that
\[
 \sum_{\substack{1\le j\le p\\\kappa(j)=k}}n_j\le m_k,\qquad 1\le k\le q,
\]
with strict inequality for at least one $k\in\{1,\ldots,q\}$ when $\B$ is unital.
\item Every spectrum-shrinking map $\A\to\B$ is spectrum-preserving if and only if $\A$ is nil.
\end{enumerate}
\end{corollary}

\begin{proof}
\textup{(a)} Here $j_\A=0$. By Theorem~\ref{thm:classification}\textup{(c)}, a continuous spectrum-preserving map $\A\to\B$ exists if and only if~\eqref{eq:rank-equations} has a solution $(\ell_j^k)$ in nonnegative integers satisfying~\eqref{eq:coverage} and, when $\B$ is unital, $\sum_{k=1}^{q}\ell_0^k>0$. Given such a solution, \eqref{eq:coverage} allows us to choose, for each $1\le j\le p$, some $\kappa(j)\in\{1,\ldots,q\}$ with $\ell_j^{\kappa(j)}>0$. We obtain
\[
 \sum_{\substack{1\le j\le p\\\kappa(j)=k}}n_j\le\sum_{j=1}^{p}n_j\ell_j^k\stackrel{\eqref{eq:zero-factor},\,\eqref{eq:rank-equations}}{=}m_k-\ell_0^k\le m_k,\qquad 1\le k\le q,
\]
with strict inequality for every $k$ with $\ell_0^k>0$, and such a $k$ exists when $\B$ is unital. Conversely, given $\kappa$ as in~\textup{(a)}, set $\ell_j^k:=1$ if $\kappa(j)=k$ and $\ell_j^k:=0$ otherwise, for $1\le j\le p$ and $1\le k\le q$, and
\[
 \ell_0^k:=m_k-\sum_{\substack{1\le j\le p\\\kappa(j)=k}}n_j,\qquad 1\le k\le q.
\]
These are nonnegative integers satisfying~\eqref{eq:rank-equations} and~\eqref{eq:coverage}, and $\sum_{k=1}^{q}\ell_0^k>0$ when $\B$ is unital, so Theorem~\ref{thm:classification}\textup{(c)} gives a polynomial spectrum-preserving map $\A\to\B$.

\textup{(b)} For every $x\in\A$, we have $0\in\spc_\A(x)$ by~\eqref{eq:zero-in-spectrum}, while $\spc_\B(0)=\{0\}$ because $\mu_0=t$. Hence the zero map $\A\to\B$ is spectrum-shrinking. If $\A$ is not nil, choose $x\in\A$ that is not nilpotent. Then $\mu_x$ is not a power of $t$, so it has a nonzero root, and the zero map does not preserve $\spc_\A(x)$. If $\A$ is nil, then every $\mu_x$, $x\in\A$, is a power of $t$, so $\spc_\A(x)=\{0\}$; since every spectrum is nonempty, every spectrum-shrinking map $\A\to\B$ is spectrum-preserving.
\end{proof}

\subsection{Scalar spectral selectors}\label{subsec:selectors}

We use the degree-one case of Lemma~\ref{lem:multiplicities} to classify the continuous spectral selectors on $\A$. A \emph{spectral selector} on $\A$ is a map $f:\A\to\C$ satisfying $f(x)\in\spc_\A(x)$ for every $x\in\A$.

By a result of Tour\'e, Schulz and Brits~\cite[Corollary~2.3]{TSB2017}, every continuous multiplicative spectral selector on a unital finite-dimensional complex associative algebra is a character. Similarly, the result of Escolano, Peralta and Villena~\cite[Theorem~3.9]{EPV2024} implies that every continuous spectral selector $f:\J\to\C$ on a unital finite-dimensional complex Jordan algebra $\J$ satisfying
\[
 f(U_x(y))=f(x)^2f(y),\qquad x,y\in\J,
\]
is a character, where $U_x$ is the quadratic operator defined by~\eqref{eq:quadratic-operator}. Indeed, $\J$ admits a complete submultiplicative norm, and its Jordan spectrum is $\spc_\J$ by Subsection~\ref{subsec:calculus}. Their theorem gives a character $\psi:\J\to\C$ such that $f=\psi$ on the connected component of $1_\J$ in the set of invertible elements. By~\eqref{eq:characteristic-spectrum}, the invertible elements form the nonvanishing set of the polynomial function $x\mapsto\chi_\J(x;0)$ from $\J$ to $\C$. Its value at $1_\J$ is nonzero by~\eqref{eq:scalar-characteristic}, so Lemma~\ref{lem:topology}\textup{(a)} shows that this set is connected and dense in $\J$. Hence $f=\psi$ on $\J$ by continuity. The following corollary requires only continuity and the spectral condition.

If a simple summand $\J_j$ in~\eqref{eq:decompositions} has Jordan rank one, then every element of $\J_j$ has a linear minimal polynomial, so $\J_j=\C 1_{\J_j}$, and we identify $\J_j$ with $\C$ via $\lambda 1_{\J_j}\mapsto\lambda$ for $\lambda\in\C$.

\begin{corollary}\label{cor:selectors}
Let $\A$ be a finite-dimensional complex power-associative algebra, with the notation of~\eqref{eq:spectral-factors}. The continuous spectral selectors $f:\A\to\C$ are exactly the coordinate maps $\pi_{\A,j}:\A\to\J_j=\C$ with $1\le j\le p$ and $n_j=1$, together with the zero map when $\A$ is nonunital. Their nonzero members are precisely the characters of $\A^+$, and they are unital when $\A$ is unital.
\end{corollary}

\begin{proof}
Let $f:\A\to\C$ be a continuous spectral selector. Apply Lemma~\ref{lem:multiplicities} to the continuous map $F:\A\to\Pol_1$ defined by $F(x;t):=t-f(x)$ for $x\in\A$, for which $Z(F(x))=\{f(x)\}\subseteq\spc_\A(x)$. Since $\sum_{j=j_\A}^{p}n_j\ell_j=1$ and $n_j\ge1$, there is a unique $j_0\in\{j_\A,\ldots,p\}$ such that $\ell_{j_0}=1$, all other exponents vanish, and $n_{j_0}=1$. If $j_0\ge1$, then
\[
 t-f(x)
 \stackrel{\eqref{eq:polynomial-multiplicities}}{=}P_{j_0}(x;t)
 \stackrel{\eqref{eq:spectral-factors},\,\eqref{eq:scalar-characteristic}}{=}
 t-\pi_{\A,j_0}(x),
 \qquad x\in\A,
\]
so $f=\pi_{\A,j_0}$. If $j_0=0$, then $\A$ is nonunital and $t-f(x)=t$ for $x\in\A$, so $f=0$.

Conversely, let $1\le j\le p$ with $n_j=1$. For $x\in\A$, \eqref{eq:scalar-characteristic} gives $P_j(x;t)=t-\pi_{\A,j}(x)$, so $\pi_{\A,j}(x)\in\spc_\A(x)$ by~\eqref{eq:spectral-reduction}. If $\A$ is nonunital, then $0\in\spc_\A(x)$ for $x\in\A$ by~\eqref{eq:spectral-reduction}, so the zero map has the property as well. All these maps are linear and therefore continuous.

Each $\pi_{\A,j}$ with $n_j=1$ is a character of $\A^+$ and is unital when $\A$ is unital. Conversely, let $\psi:\A^+\to\C$ be a character. It is linear, hence continuous. If $\A$ is unital, then $\psi(1_\A)\psi(x)=\psi(x)$ for every $x\in\A$, and $\psi\ne0$ gives $\psi(1_\A)=1$. In either case, \eqref{eq:hom-shrinking}, applied with $\B=\C$, gives $\psi(x)\in\spc_\A(x)$ for every $x\in\A$, and the first part of the proof identifies $\psi$ with a rank-one coordinate map.
\end{proof}

\subsection{Associative algebras}\label{subsec:associative}

For associative algebras, the semisimple Jordan quotient is the symmetrization of the classical semisimple quotient. This identifies the data of Theorem~\ref{thm:classification} with those of~\cite{GT2025} and settles the question left open there.

\begin{lemma}\label{lem:associative}
Let $\A$ be a unital finite-dimensional complex associative algebra, with the Jordan ranks in~\eqref{eq:spectral-factors}, and let $\rad(\A)$ denote its Jacobson radical. Then $\Nil(\A^+)=\rad(\A)$, and the maximal ideals of $\A$ may be labeled $\mathcal M_1,\ldots,\mathcal M_p$ so that
\[
 \A/\mathcal M_j\cong M_{n_j}(\C),\qquad 1\le j\le p.
\]
\end{lemma}

\begin{proof}
By Wedderburn's structure theorem, $\A/\rad(\A)$ is the direct sum of the quotients of $\A$ by its maximal ideals, and these quotients are full matrix algebras; see~\cite[\S\,2.7]{Bresar2025}. The radical $\rad(\A)$ is nilpotent, so it is a nil ideal of $\A^+$ and $\rad(\A)\subseteq\Nil(\A^+)$. The quotient $\A^+/\rad(\A)\cong(\A/\rad(\A))^+$ is a direct sum of full matrix Jordan algebras, each simple and unital by Section~\ref{subsec:models}.

The image of $\Nil(\A^+)$ in each simple summand of this quotient is a nil ideal. It cannot contain the identity of the corresponding summand, so it is zero by simplicity. Hence $\Nil(\A^+)\subseteq\rad(\A)$, proving equality.

In a finite direct sum of simple unital Jordan algebras, the summands are precisely the minimal nonzero ideals: multiplication by the identity of a summand extracts its coordinate. Consequently, the simple summands of $\widehat\A$ are, up to order and isomorphism, the symmetrizations of the simple associative quotients. Their Jordan ranks are the corresponding matrix sizes by Section~\ref{subsec:models}, so the maximal ideals can be labeled as asserted.
\end{proof}

\begin{corollary}\label{cor:associative}
Let $\A$ and $\B$ be unital finite-dimensional complex associative algebras, with the notation of~\eqref{eq:spectral-factors}. Then every continuous spectrum-shrinking map $\A\to\B$ is spectrum-preserving if and only if every solution of~\eqref{eq:rank-equations} in nonnegative integers satisfies~\eqref{eq:coverage}. In particular, the converse of~\cite[Theorem~1.1(c)]{GT2025} holds for every unital finite-dimensional complex associative source $\A$.
\end{corollary}

\begin{proof}
By Lemma~\ref{lem:associative}, the integers $n_j$, $1\le j\le p$, and $m_k$, $1\le k\le q$, are the sizes of the simple matrix quotients of $\A$ and $\B$, respectively, or equivalently, the square roots of the quotient dimensions. Thus~\eqref{eq:rank-equations} and~\eqref{eq:coverage} are precisely the conditions of~\cite[Theorem~1.1]{GT2025}. The asserted equivalence now follows from Theorem~\ref{thm:classification}\textup{(b)} and~\textup{(c)}.
\end{proof}

For a unital associative algebra $\A$, Lemma~\ref{lem:associative} identifies the rank-one summands of $\widehat\A$ with the ideals of $\A$ of codimension one. For $1\le j\le p$ with $n_j=1$, the coordinate map $\pi_{\A,j}:\A\to\C$ is the quotient homomorphism $\A\to\A/\ker\pi_{\A,j}\cong\C$. Thus Corollary~\ref{cor:selectors} recovers~\cite[Corollary~2.2]{GT2025}: a continuous spectral selector exists if and only if $\A$ has an ideal of codimension one. It also shows that every continuous spectral selector is one of these quotient homomorphisms.

\subsection{Examples}\label{subsec:examples}

We conclude with examples illustrating the roles of symmetrization and continuity in Corollary~\ref{cor:selectors}: a continuous spectral selector need not be multiplicative for the original product, and a discontinuous spectral selector need not be additive. We also exhibit a continuous spectrum-preserving map whose dependence on the nilradical coordinate is not polynomial.

\begin{example}\label{ex:symmetrization}
Let $\mathcal T$ be the three-dimensional algebra of Kosier~\cite[p.~306]{Kosier1962}, with basis $\{1,u,v\}\subseteq\mathcal T$, identity $1$, and remaining products of basis elements
\[
u^2:=v^2:=0, \qquad uv:=1 \qquad \text{and} \qquad vu:=-1.
\]
For $w\in\C u\oplus\C v$ we have $w^2=0$, so $\C 1+\C w$ is a homomorphic image of $\C[t]/(t^2)$ and hence associative. Since every element of $\mathcal T$ lies in such a subalgebra, $\mathcal T$ is power-associative. Note also that $\mathcal T$ fails the flexible identity $(xy)x=x(yx)$, $x,y\in\mathcal T$, since $(uv)u=u\ne-u=u(vu)$.

The algebra $\mathcal T$ is simple. Indeed, let $\mathcal I\subseteq\mathcal T$ be a nonzero ideal and $x:=\alpha1+\beta u+\gamma v\in\mathcal I\setminus\{0\}$ with $\alpha,\beta,\gamma\in\C$. Direct multiplication gives $xu-ux=-2\gamma 1$ and $xv-vx=2\beta 1$. If $\beta\ne0$ or $\gamma\ne0$, then $1\in\mathcal I$; otherwise $x=\alpha1$ with $\alpha\ne0$, and again $1\in\mathcal I$. Hence $\mathcal I=\mathcal T$.

In $\mathcal T^+$, however, $\C u\oplus\C v$ is an ideal with zero product, and the quotient is $\C$. Every nil ideal of $\mathcal T^+$ has zero image in this quotient, since $\C$ has no nonzero nilpotent elements. Therefore $\Nil(\mathcal T^+)=\C u\oplus\C v$ and $\widehat{\mathcal T}=\C$. By Corollary~\ref{cor:selectors}, the unique continuous spectral selector on $\mathcal T$ is
\[
 f:\mathcal T\to\C,\qquad f(\alpha1+\beta u+\gamma v):=\alpha,\qquad\alpha,\beta,\gamma\in\C.
\]
It is a character of $\mathcal T^+$ but is not multiplicative on $\mathcal T$, since $f(uv)=1$ and $f(u)f(v)=0$.
\end{example}

\begin{example}
Let $f:M_2(\C)\to\C$ assign to each matrix the eigenvalue with largest real part and, among these, largest imaginary part. Then $f$ is a spectral selector. Put $A:=\diag(i,-i)\in M_2(\C)$ and, for $k\in\N$, set $A_k:=\diag(-1/k+i,-i)\in M_2(\C)$. Then $\lim_{k \to \infty} A_k=A$, but $f(A_k)=-i$ for every $k\in\N$, whereas $f(A)=i$. Thus $f$ is discontinuous at $A$. It is also not additive, since $f(\diag(1,0))+f(\diag(0,1))=2$, whereas $f(I_2)=1$.
\end{example}

\begin{example}\label{ex:nil-dependence}
Let $\mathcal D:=\C[t]/(t^2)$, and let $z:=t+(t^2)\in\mathcal D$. Define a map $\phi:\mathcal D\to M_2(\C)$ by
\[
 \phi(\alpha1+\beta z):=\begin{bmatrix}\alpha&\overline\beta\\0&\alpha\end{bmatrix},\qquad \alpha,\beta\in\C,
\]
where the bar denotes complex conjugation. The map $\phi$ is continuous, and both $\spc_{\mathcal D}(\alpha1+\beta z)$ and $\spc_{M_2(\C)}(\phi(\alpha1+\beta z))$ equal $\{\alpha\}$, so $\phi$ is spectrum-preserving. Since $\Nil(\mathcal D^+)=\C z$ and $\phi$ depends on $\beta$, it does not factor through the quotient map $\mathcal D\to\widehat{\mathcal D}$. Moreover, $\phi$ is not polynomial. Otherwise, restricting its upper-right entry to $\alpha=0$ would give a polynomial $g\in\C[t]$ with $g(\beta)=\overline\beta$ for every $\beta\in\C$. This polynomial agrees with $t$ at all positive integers, so $g=t$, whereas $\overline i=-i\ne i$.
\end{example}

\subsection*{Funding}
This research was supported by the European Union -- NextGenerationEU through the National Recovery and Resilience Plan 2021--2026, Institutional Grants of the University of Zagreb Faculty of Science (IK IA 1.1.3. Impact4Math, PMF-CROFUND).

\subsection*{AI use disclosure}
The author used OpenAI's GPT-5.6 and subsequently GPT-6.1 Sol for preliminary exploration of proof strategies, literature searches, suggestions on proof simplification and exposition, copyediting and proofreading. In particular, GPT-6.1 Sol suggested replacing the companion-matrix argument in Step~1 of the proof of Proposition~\ref{prop:sections} with a Hessenberg--Toeplitz construction. It also suggested using Rouch\'e's theorem in place of the complex implicit function theorem in Step~2 of the proof of Lemma~\ref{lem:multiplicities}. Both changes make the arguments more direct. Anthropic's Claude Opus 5.5 was also used for literature searches and proofreading. The author reviewed and verified all AI-assisted material and takes full responsibility for the final content.


\begin{thebibliography}{99}

\bibitem{Albert1948}
A.~A.~Albert,
\emph{Power-associative rings},
Trans. Amer. Math. Soc. 64 (1948), no.~3, 552--593,
doi: \href{https://doi.org/10.1090/S0002-9947-1948-0027750-7}{10.1090/S0002-9947-1948-0027750-7}.

\bibitem{Albert1950}
A.~A.~Albert,
\emph{A theory of power-associative commutative algebras},
Trans. Amer. Math. Soc. 69 (1950), no.~3, 503--527,
doi: \href{https://doi.org/10.1090/S0002-9947-1950-0038959-X}{10.1090/S0002-9947-1950-0038959-X}.

\bibitem{Aupetit2000}
B.~Aupetit,
\emph{Spectrum-preserving linear mappings between Banach algebras or Jordan--Banach algebras},
J. London Math. Soc. (2) 62 (2000), no.~3, 917--924,
doi: \href{https://doi.org/10.1112/S0024610700001514}{10.1112/S0024610700001514}.

\bibitem{BaribeauRansford2000}
L.~Baribeau and T.~Ransford,
\emph{Non-linear spectrum-preserving maps},
Bull. London Math. Soc. 32 (2000), no.~1, 8--14,
doi: \href{https://doi.org/10.1112/S0024609399006426}{10.1112/S0024609399006426}.

\bibitem{Bresar2025}
M.~Bre\v{s}ar,
\emph{Introduction to noncommutative algebra},
2nd ed., Universitext, Springer, Cham, 2025,
doi: \href{https://doi.org/10.1007/978-3-031-96296-7}{10.1007/978-3-031-96296-7}.

\bibitem{CLNT2017}
M.-T.~Chien, J.~Liu, H.~Nakazato and T.-Y.~Tam,
\emph{Toeplitz matrices are unitarily similar to symmetric matrices},
Linear Multilinear Algebra 65 (2017), no.~10, 2131--2144,
doi: \href{https://doi.org/10.1080/03081087.2017.1330865}{10.1080/03081087.2017.1330865}.

\bibitem{Chirvasitu2025a}
A.~Chirvasitu,
\emph{Eigenvalue selectors for representations of compact connected groups},
preprint (2025), arXiv:2502.08847,
doi: \href{https://doi.org/10.48550/arXiv.2502.08847}{10.48550/arXiv.2502.08847}.

\bibitem{Chirvasitu2025b}
A.~Chirvasitu,
\emph{Spectral selections, commutativity preservation and Coxeter--Lipschitz maps},
preprint (2025), arXiv:2505.19393,
doi: \href{https://doi.org/10.48550/arXiv.2505.19393}{10.48550/arXiv.2505.19393}.

\bibitem{Chirvasitu2026a}
A.~Chirvasitu,
\emph{The combinatorics of permuting and preserving curve-bound spectra},
preprint (2026), arXiv:2601.01208,
doi: \href{https://doi.org/10.48550/arXiv.2601.01208}{10.48550/arXiv.2601.01208}.

\bibitem{Chirvasitu2026b}
A.~Chirvasitu,
\emph{Eigenvalue collision and exotic preservers on semisimple operators},
preprint (2026), arXiv:2604.07062,
doi: \href{https://doi.org/10.48550/arXiv.2604.07062}{10.48550/arXiv.2604.07062}.

\bibitem{CGT2025}
A.~Chirvasitu, I.~Gogi\'{c} and M.~Toma\v{s}evi\'{c},
\emph{A variant of \v{S}emrl's preserver theorem for singular matrices},
Linear Algebra Appl. 724 (2025), 298--319,
doi: \href{https://doi.org/10.1016/j.laa.2025.06.021}{10.1016/j.laa.2025.06.021}.

\bibitem{CGT2026}
A.~Chirvasitu, I.~Gogi\'{c} and M.~Toma\v{s}evi\'{c},
\emph{Continuous spectrum-shrinking maps and applications to preserver problems},
Proc. Edinb. Math. Soc. 69 (2026), no.~4, 1340--1371,
doi: \href{https://doi.org/10.1017/S0013091526101473}{10.1017/S0013091526101473}.

\bibitem{Costara2011}
C.~Costara,
\emph{Maps on matrices that preserve the spectrum},
Linear Algebra Appl. 435 (2011), no.~11, 2674--2680,
doi: \href{https://doi.org/10.1016/j.laa.2011.04.026}{10.1016/j.laa.2011.04.026}.

\bibitem{CuiHou2004}
J.~L.~Cui and J.~C.~Hou,
\emph{A characterization of homomorphisms between Banach algebras},
Acta Math. Sin. (Engl. Ser.) 20 (2004), no.~4, 761--768,
doi: \href{https://doi.org/10.1007/s10114-004-0312-8}{10.1007/s10114-004-0312-8}.

\bibitem{DHKQ2013}
G.~Dolinar, J.~Hou, B.~Kuzma and X.~Qi,
\emph{Spectrum nonincreasing maps on matrices},
Linear Algebra Appl. 438 (2013), no.~8, 3504--3510,
doi: \href{https://doi.org/10.1016/j.laa.2012.12.034}{10.1016/j.laa.2012.12.034}.

\bibitem{EPV2024}
G.~M.~Escolano, A.~M.~Peralta and A.~R.~Villena,
\emph{Lie--Trotter formulae in Jordan--Banach algebras with applications to the study of spectral-valued multiplicative functionals},
Results Math. 79 (2024), no.~1, article 17,
doi: \href{https://doi.org/10.1007/s00025-023-02043-w}{10.1007/s00025-023-02043-w}.

\bibitem{Fiedler1990}
M.~Fiedler,
\emph{Expressing a polynomial as the characteristic polynomial of a symmetric matrix},
Linear Algebra Appl. 141 (1990), 265--270,
doi: \href{https://doi.org/10.1016/0024-3795(90)90323-5}{10.1016/0024-3795(90)90323-5}.

\bibitem{Frobenius1897}
G.~Frobenius,
\emph{\"Uber die Darstellung der endlichen Gruppen durch lineare Substitutionen},
Sitzungsber. Preuss. Akad. Wiss. Berlin (1897), 994--1015,
digital facsimile doi: \href{https://doi.org/10.3931/e-rara-18879}{10.3931/e-rara-18879}.

\bibitem{Gleason1967}
A.~M.~Gleason,
\emph{A characterization of maximal ideals},
J. Analyse Math. 19 (1967), 171--172,
doi: \href{https://doi.org/10.1007/BF02788714}{10.1007/BF02788714}.

\bibitem{GT2025}
I.~Gogi\'{c} and M.~Toma\v{s}evi\'{c},
\emph{Continuous spectrum-shrinking maps between finite-dimensional algebras},
J. Algebra Appl., published online (2025), article 2650298,
doi: \href{https://doi.org/10.1142/S0219498826502981}{10.1142/S0219498826502981}.

\bibitem{Goldberg2007}
M.~Goldberg,
\emph{Minimal polynomials and radii of elements in finite-dimensional power-associative algebras},
Trans. Amer. Math. Soc. 359 (2007), no.~8, 4055--4072,
doi: \href{https://doi.org/10.1090/S0002-9947-06-04296-6}{10.1090/S0002-9947-06-04296-6}.

\bibitem{Jacobson1963}
N.~Jacobson,
\emph{Generic norm of an algebra},
Osaka Math. J. 15 (1963), no.~1, 25--50,
doi: \href{https://doi.org/10.18910/8311}{10.18910/8311}.

\bibitem{Jacobson1968}
N.~Jacobson,
\emph{Structure and representations of Jordan algebras},
American Mathematical Society Colloquium Publications, Vol.~39,
American Mathematical Society, Providence, RI, 1968,
doi: \href{https://doi.org/10.1090/coll/039}{10.1090/coll/039}.

\bibitem{KahaneZelazko1968}
J.-P.~Kahane and W.~\.{Z}elazko,
\emph{A characterization of maximal ideals in commutative Banach algebras},
Studia Math. 29 (1968), no.~3, 339--343,
doi: \href{https://doi.org/10.4064/sm-29-3-339-343}{10.4064/sm-29-3-339-343}.

\bibitem{Kaplansky1970}
I.~Kaplansky,
\emph{Algebraic and analytic aspects of operator algebras},
CBMS Regional Conference Series in Mathematics, Vol.~1,
American Mathematical Society, Providence, RI, 1970,
doi: \href{https://doi.org/10.1090/cbms/001}{10.1090/cbms/001}.

\bibitem{Kokoris1956}
L.~A.~Kokoris,
\emph{Simple power-associative algebras of degree two},
Ann. of Math. (2) 64 (1956), no.~3, 544--550,
doi: \href{https://doi.org/10.2307/1969601}{10.2307/1969601}.

\bibitem{Kosier1962}
F.~Kosier,
\emph{On a class of nonflexible algebras},
Trans. Amer. Math. Soc. 102 (1962), no.~2, 299--318,
doi: \href{https://doi.org/10.1090/S0002-9947-1962-0133353-3}{10.1090/S0002-9947-1962-0133353-3}.

\bibitem{KostantRallis1971}
B.~Kostant and S.~Rallis,
\emph{Orbits and representations associated with symmetric spaces},
Amer. J. Math. 93 (1971), no.~3, 753--809,
doi: \href{https://doi.org/10.2307/2373470}{10.2307/2373470}.

\bibitem{KowalskiSlodkowski1980}
S.~Kowalski and Z.~S\l{}odkowski,
\emph{A characterization of multiplicative linear functionals in Banach algebras},
Studia Math. 67 (1980), no.~3, 215--223,
doi: \href{https://doi.org/10.4064/sm-67-3-215-223}{10.4064/sm-67-3-215-223}.

\bibitem{MMP1999}
D.~S.~Mackey, N.~Mackey and S.~Petrovi\'c,
\emph{Is every matrix similar to a Toeplitz matrix?},
Linear Algebra Appl. 297 (1999), no.~1--3, 87--105,
doi: \href{https://doi.org/10.1016/S0024-3795(99)00131-7}{10.1016/S0024-3795(99)00131-7}.

\bibitem{MarcosVelasco2014}
J.~C.~Marcos and M.~V.~Velasco,
\emph{The multiplicative spectrum and the uniqueness of the complete norm topology},
Filomat 28 (2014), no.~3, 473--485,
doi: \href{https://doi.org/10.2298/FIL1403473M}{10.2298/FIL1403473M}.

\bibitem{MarcusPurves1959}
M.~Marcus and R.~Purves,
\emph{Linear transformations on algebras of matrices: the invariance of the elementary symmetric functions},
Canad. J. Math. 11 (1959), 383--396,
doi: \href{https://doi.org/10.4153/CJM-1959-039-4}{10.4153/CJM-1959-039-4}.

\bibitem{MathieuSchulz2023}
M.~Mathieu and F.~Schulz,
\emph{Invertibility preserving mappings onto finite $C^*$-algebras},
Studia Math. 272 (2023), no.~1, 113--120,
doi: \href{https://doi.org/10.4064/sm230101-27-3}{10.4064/sm230101-27-3}.

\bibitem{McCrimmon2004}
K.~McCrimmon,
\emph{A taste of Jordan algebras},
Universitext, Springer, New York, 2004,
doi: \href{https://doi.org/10.1007/b97489}{10.1007/b97489}.

\bibitem{Schafer1966}
R.~D.~Schafer,
\emph{An introduction to nonassociative algebras},
Pure and Applied Mathematics, Vol.~22,
Academic Press, New York--London, 1966,
doi: \href{https://doi.org/10.1016/S0079-8169(08)X6123-9}{10.1016/S0079-8169(08)X6123-9}.

\bibitem{Semrl2008}
P.~\v{S}emrl,
\emph{Characterizing Jordan automorphisms of matrix algebras through preserving properties},
Oper. Matrices 2 (2008), no.~1, 125--136,
doi: \href{https://doi.org/10.7153/oam-02-07}{10.7153/oam-02-07}.

\bibitem{Sourour1996}
A.~R.~Sourour,
\emph{Invertibility preserving linear maps on $\mathcal L(X)$},
Trans. Amer. Math. Soc. 348 (1996), no.~1, 13--30,
doi: \href{https://doi.org/10.1090/S0002-9947-96-01428-6}{10.1090/S0002-9947-96-01428-6}.

\bibitem{TSB2017}
C.~Tour\'e, F.~Schulz and R.~Brits,
\emph{Multiplicative maps into the spectrum},
Studia Math. 239 (2017), no.~1, 55--66,
doi: \href{https://doi.org/10.4064/sm8705-1-2017}{10.4064/sm8705-1-2017}.

\bibitem{Zelazko1968}
W.~\.{Z}elazko,
\emph{A characterization of multiplicative linear functionals in complex Banach algebras},
Studia Math. 30 (1968), no.~1, 83--85,
doi: \href{https://doi.org/10.4064/sm-30-1-83-85}{10.4064/sm-30-1-83-85}.

\end{thebibliography}
\end{document}